\documentclass{amsart}
\usepackage[textsize=scriptsize,colorinlistoftodos]{todonotes}

\usepackage[margin=1in]{geometry}
\usepackage{etoolbox}

\usepackage{longtable}

\newtoggle{final}

\usepackage{tikz}
\usetikzlibrary{matrix,arrows}
\usepackage{tikz-cd}

\usepackage{amsmath}
\usepackage{amssymb}
\usepackage{amsthm}
\usepackage{thmtools}
\usepackage{amscd}
\usepackage{enumerate}
\usepackage[pdfusetitle,unicode,hidelinks,hypertexnames=false]{hyperref}
\usepackage{bbm}
\usepackage{etoolbox}
\usepackage{hyperref}
\usepackage{quiver}

\usepackage[backend=biber,style=alphabetic]{biblatex}
\usepackage[utf8]{inputenc}
\usepackage[T1]{fontenc}

\newcommand{\timonemail}{\href{mailto:tt@math.ku.dk}{tt@math.ku.dk}}

\newtheorem{proposition}{Proposition}[section]
\newtheorem{corollary}[proposition]{Corollary}
\newtheorem{lemma}[proposition]{Lemma}
\newtheorem{theorem}[proposition]{Theorem}

\newtheorem*{conjecture*}{Conjecture}
\newtheorem*{theorem*}{Theorem}
\newtheorem*{corollary*}{Corollary}
\newtheorem*{proposition*}{Proposition}
\newtheorem*{lemma*}{Lemma}
\theoremstyle{definition}
\newtheorem{definition}[proposition]{Definition}
\newtheorem{construction}[proposition]{Construction}

\newtheorem*{definition*}{Definition}
\newtheorem*{construction*}{Construction}
\theoremstyle{remark}
\newtheorem{remark}[proposition]{Remark}
\newtheorem*{remark*}{Remark}

\newtheorem{example}[proposition]{Example}
\newtheorem*{example*}{Example}

\newcommand{\Z}{\mathbb{Z}}
\def\C{\mathcal C}
\newcommand{\N}{\mathbb{N}}

\newcommand{\LM}{\mathrm{L}_{\mathrm{mot}}}

\let\bb=\mathbb
\newcommand{\Gm}{\mathbb{G}_m}

\def\A{\bb A}

\newcommand{\SH}[1]{\mathcal{SH}(#1)}

\let\lim=\relax
\DeclareMathOperator*{\lim}{lim}
\def\CAlg{\mathrm{CAlg}}

\def\PSh{\mathcal{P}}

\def\Span{\mathrm{Span}}
\def\Cat{\mathrm{Cat}{}}

\newcommand{\Spec}{\mathrm{Spec}}

\let\cat=\mathrm
\def\Sm{{\cat{S}\mathrm{m}}}
\def\Sch{\cat{S}\mathrm{ch}{}}

\newcommand{\HZ}{H\mathbb{Z}}
\newcommand{\HZp}{\HZ^\wedge_p}
\newcommand{\KM}[2]{K^{\mathrm{M}}_{#1}\!\left(#2\right)}
\newcommand{\HS}[3]{\underline{\pi}_{#1}(#3)_{#2}}
\newcommand{\HSeff}[2]{\underline{\pi}^{\mathrm{eff}}_{#1}(#2)}
\newcommand{\KMS}[1]{\underline{k}^{\mathrm{M}}_{#1}}
\newcommand{\Sph}{\mathbb{S}}
\newcommand{\Spc}[1]{\mathcal{S}\mathrm{pc}(#1)}
\newcommand{\MGL}{\mathrm{MGL}}
\newcommand{\PMGL}{\mathrm{PMGL}}

\newcommand{\sus}{\Sigma_+^{\infty}}

\newcommand{\Bet}[1]{\mathrm{B}_{\text{ét}}#1}
\newcommand{\BGm}{\Bet{\Gm}}
\newcommand{\Proj}{\mathbb{P}}

\newcommand{\An}{\mathrm{Ani}}

\DeclareMathOperator*{\limit}{lim}
\newcommand{\ilim}[2]{\limit_{#1}\left(#2\right)}
\DeclareMathOperator*{\colimit}{colim}
\newcommand{\colim}[2]{\underset{#1}{\colimit}\left(#2\right)}

\newcommand{\Map}[3]{\mathrm{Map}_{#1}\left(#2,#3\right)}
\newcommand{\map}[3]{\mathrm{map}_{#1}\left(#2,#3\right)}

\newcommand{\fib}{\operatorname{fib}}

\newcommand{\sslash}{/\mkern-6mu/}

\usepackage[capitalise]{cleveref}
\usepackage{mathtools}
\usepackage{dsfont}
\theoremstyle{plain}

\crefname{theorem}{Theorem}{Theorems}
\crefname{introtheorem}{Theorem}{Theorems}
\crefname{lemma}{Lemma}{Lemmas}
\crefname{construction}{Construction}{Constructions}
\crefname{proposition}{Proposition}{Propositions}

\newcommand{\comp}{{{\kern -.5pt}\wedge}}

\iftoggle{final}{
	\newcommand{\NB}[1]{}
	\newcommand{\todolater}[1]{}

	\renewcommand{\todo}[1]{}
	
	\newcommand{\todoklaus}[1]{}
	\newcommand{\todotom}[1]{}
}{ 
	\newcommand{\NB}[1]{\todo[color=gray!40]{#1}}
	
	\newcommand{\todolater}[1]{\todo[color=brown!40]{#1}}

	\newcommand{\todoklaus}[1]{\todo[color=blue]{#1}}
	\newcommand{\todotom}[1]{\todo[color=green]{#1}}
}

\title{\texorpdfstring{Motivic $C_p$-power operations in $\MGL$}{Motivic Cp-power operations in MGL}}
\date{\today}

\author{Timon Tausendpfund}
\email{\timonemail}
\urladdr{\href{https://ttausendkp.github.io/}{ttausendkp.github.io}}

\begin{document}

	\begin{abstract}
		In this paper, we give a new proof of the Hopkins-Morel-Hoyois theorem for fields $k$ of characteristic unequal to $p$ containing all primitive $p$-power roots of unity for a fixed prime $p$. This new proof follows Quillen's proof for the structure of the homotopy groups of $\mathrm{MU}$ which utilizes topological power operations. To this end we define a motivic analog of the Steenrod power operations, using the theory of normed motivic spectra as developed by Bachmann--Hoyois and their application for the definition of motivic colimits as developed by Bachmann-Elmanto-Heller.
	\end{abstract}
	\maketitle
	\tableofcontents

	\section{Introduction}

The algebraic cobordism spectrum $\MGL$ is one of the central objects in motivic homotopy theory. Just as in algebraic topology, a map $\MGL\to E$ of homotopy commutative ring spectra defines an \emph{orientation} on $E$. This definition can be used to develop a theory of orientations on $E$, as done for example in \cite{degliseorientation}. Orientation theory is an important tool for the study of vector bundles over schemes.

Thus, just as in algebraic topology, we would like to know what the homotopy groups $\pi_\ast(\MGL)_\ast$ are. This was almost achieved by Marc Hoyois in \cite{hoy15}, where he reverse engineered the calculation from talks of Fabien Morel and Mike Hopkins. He proves that, if $e$ is the exponential characteristic of a field $K$, then there is an equivalence of motivic spectra\footnote{See the introduction to \cite{hoy15} for the relation between the equivalence of spectra and the the homotopy groups we are interested in.}
\begin{equation*}
	\MGL/(a_1,a_2,\ldots)[1/e]\simeq H\Z[1/e].
\end{equation*}
Here the $a_i$ are a set of Lazard generators. In particular, over fields of characteristic $0$, the full answer is known. The problems, however, are field of characteristic $p$, where the full answer is still unknown.

In this paper we give an alternative proof of the Hopkins-Morel-Hoyois theorem over fields admitting all primitive $p$-th roots of unity. This new proof depends on the additional structure of $\MGL$ being a \emph{normed} Motivic spectrum, that is, $\MGL$ admitting multiplicative norms along finite étale morphisms. Using these norms, we can define the notion of motivic power operations. This enables us to copy quite closely the original proof of Quillen from \cite{BordismQuillen}. What makes this proof interesting is that it carries the potential for generalization. For example, if it would be possible to extend the transfers on the stable motivic homotopy category $\SH{S}$ for some qcqs scheme $S$ to finite flat morphisms, the proof we give could be extended to a more general setting. Furthermore, we are currently working on reducing on the number of primitive $p$-th power roots in the assumption.

\paragraph{\textbf{Outline of the proof}}
The proof proceeds in essentially three steps. In the first step, we explicitly calculate the motivic $C_p$-power operations on $\MGL$ and use this to give an explicit description of $p$-th powers of elements in $\MGL^{2\ast,\ast}$.

In a second step we explicitly calculate the bigraded homotopy groups $\pi_\ast(\MGL^\wedge_p)_\ast$. This happens by a two-fold induction. First we show that if we know the result for bidegrees $(s-1,\ast)$, then we also know the answer for the group $\pi_s(\MGL^\wedge_p)_0$. This uses the aforementioned control over $p$-th power of elements in the bigraded homotopy groups. In a second inductive step we show that, once we know the group $\pi_s(\MGL^\wedge_p)_0$, we know the whole diagonal $(s,\ast)$. This uses that $\MGL$ is an effective spectrum.

In the forth step, one uses the knowledge of the rationalization of $\MGL$ and the arithmetic fracture square to deduce the result for $\MGL$ itself. As we do not have anything new to add for the last step, we will not discuss this in this paper.

\paragraph{\textbf{Conventions}}
\begin{itemize}
	\item We freely use the language of $\infty$-categories, which we will hereafter simply call categories.
	\item We will denote the category of Anima or spaces by $\An$. Further, we denote the category of categories by $\mathrm{Cat}$.
	\item For a category $\C$, we write $\mathcal{P}(\C)=\mathrm{Fun}(\C^{\mathrm{op}},\An)$ for the free cocompletion of $\C$, also known as the presheaf category.
	\item We assume the reader to be familiar with terms from algebraic geometry. To fix notation, for us a scheme will always be a quasi-compact, quasi-separated (qcqs), if not explicitly stated otherwise. We write $\mathrm{Sch}$ for the category of all (qcqs) schemes, $\mathrm{Sch}_{S}:=\mathrm{Sch}_{/S}$ for the category of $S$-schemes for a fixed scheme $S$, $\mathrm{Sm}(S)$ for the subcategory of smooth schemes over a fixed base scheme $S$, $\mathrm{Sm}^{\mathrm{ft}}(S)$ for the full subcategory of smooth schemes of finite type, which is an essentially small category.
	\item We further assume the reader to be familiar with the notions of motivic homotopy theory. We refer to \cite{hoyois2017six} for an introduction to this topic.
\end{itemize}
\paragraph{\textbf{Aknowledgment}:} This paper is based on a master's thesis supervised by Prof.\! Dr.\! Tom Bachmann at the Johannes-Gutenberg University Mainz. I would like to express my deepest gratitude to him for his continuous support during and after writing the thesis. Further, I would like to thank my second supervisor Prof.\! Dr.\! Marc Hoyois for helpful discussions, in particular his explanations about the constructions of normed categories. I would also like to express my gratitude to Dr.\! Lorenzo Mantovani, Luca Passolunghi, and Prof.\! Dr.\! Georg Tamme for many helpful discussions. Finally, I acknowledge the financial support provided by the deutsche Forschungsgemeinschaft (DFG, German Research Foundation) through the Collaborative Research Center TRR 326 GAUS Geometry and Arithmetic of Uniformized Structures), and the support by the European Union’s Horizon 2020 research and innovation programme under the European Research Council (ERC) grant agreement No. 101164148 and by the Danish National Research Foundation through the Copenhagen Centre for Geometry and Topology (DNRF151) under which this project was finalized.
	\section{Motivic power operations}
In this introductory section we are going to define the motivic power operations, the main ingredient in our new proof of the Hopkins-Morel-Hoyois theorem.

\subsection{The normed motivic stable homotopy category}
Recall from \cite[Section 6]{NormsBachmannHoyois} that the motivic stable homotopy category $\SH{-}:\Sch^{\mathrm{op}}\to \mathrm{Pr}^L$ can be extended to a normed category
\begin{equation*}
	\SH{-}^{\otimes}:\Span(\Sch,\mathrm{all},\mathrm{f\acute{e}t})\to \mathrm{Pr}^L,
\end{equation*} 
where $\Span(\Sch,\mathrm{all},\mathrm{f\acute{e}t})$ denotes the $2$-category of spans in schemes, where we allow all morphisms as left pointing ones, and only finite étale morphisms as right pointing ones. See for example \cite{hauhebetwovariab} for a rigorous definition of this category.

This additional structure allows the definition of the following functor, see \cite{MotivicColimits} and \cite{bachmannMMF} for details.
\begin{definition}\label{def: extended powers}
Let $S$ be a qcqs-scheme and $\Sigma_n$ the (finite) symmetric group on $n$-letters.
\begin{itemize}
	\item[1)] We define the \emph{enhanced $n$-fold smash powers} as the composition
	\begin{equation*}
		(-)^{\wedge\underline{n}}:\SH{S}\xrightarrow{\mathrm{triv}}\SH{S}^{\Sigma_n}\xrightarrow{p^\ast}\SH{S^{\sqcup n}}^{\Sigma_n}\xrightarrow{p_\otimes }\SH{S}^{\Sigma_n}.
	\end{equation*}
	Here, $p:S^{\amalg n}\to S$ is the codiagonal, which we consider as an $\Sigma_n$ equivariant map with respect to the permutation action on $S^{\amalg n}$ and the trivial action on $S$, $p^\ast$ is the induced pullback functor and $p_\otimes$ is the equivariant norms functor, see \cite[Corollary 3.33]{bachmannMMF}. Further, $\mathrm{triv}$ is the the functor endowing a motivic spectrum with the trivial action. Using the description of $\SH{S}^G\simeq \LM(\PSh{\mathrm{Sm}_S^G})[\mathrm{Sph}^{-1}_S]$, see \cite[Corollary 6.7]{hoyois2017six}, this functor is induced by the canonical functor $\Sm_S\to\Sm_S^G$ sending a scheme $X$ to the scheme with the trivial $G$-action, see also \cite[Appendix C.4]{MotivicColimits}.
	\item[2)] We define the \emph{geometric homotopy orbits functor} as the composition
	\begin{equation*}
		(-)_{\mathds{h}G}:=\SH{S}^G\xrightarrow{(-)\otimes (\mathbb{E}G)_+}\SH{S}^G[\mathcal{F}_{\mathrm{triv}}]\xrightarrow{(-)/G}\SH{S}.
	\end{equation*} 
	Here $(\mathbb{E}G)_+:=\Sigma^\infty_+\mathbb{E}_G$ is the geometric universal $G$-space, $\SH{S}^G[\mathcal{F}_{\mathrm{triv}}]$ is the full subcategory of generated by the spheres and those suspension spectra corresponding to $G$-schemes with trivial stabilizer. Lastly, $(-)/G$ is the left adjoint of $(-)\otimes (\mathbb{E}G)_+\circ \mathrm{triv}$, see \cite[Proposition C.5 (2)]{MotivicColimits}.
\end{itemize}
Let now  $G\subseteq \Sigma_n$ be a subgroup.
\begin{itemize}
	\item[3)] We define the \emph{geometric homotopy orbits functor} as the composition
	\begin{equation*}
		(-)_{\mathds{h}G}:=\SH{S}^G\xrightarrow{(-)\otimes (\mathbb{E}G)_+}\SH{S}^G[\mathcal{F}_{\mathrm{triv}}]\xrightarrow{(-)/G}\SH{S}.
	\end{equation*} 
	
	\item[4)] We define the \emph{geometric extend power functor} as the composite
	\begin{equation*}
		\mathbb{D}_G:\SH{S}\xrightarrow{(-)^{\underline{\wedge}n}}\SH{S}^{\Sigma_n}\xrightarrow{\mathrm{res}}\SH{S}^G\xrightarrow{(-)_{\mathds{h}G}}\SH{S}.
	\end{equation*}
	If $G=\Sigma_n$, we write simply $\mathbb{D}_{\Sigma_n}=:\mathbb{D}_n$.
\end{itemize}
\end{definition}
We summarize the most important properties of the motivic power operations and the geometric homotopy orbits for us here.
\begin{proposition}
	Let $G\subseteq \Sigma_n$ be a finite group, $S$ be a scheme.
	\begin{itemize}
		\item[1)] Let $E\in\SH{S}$ and $F\in\SH{S}^G$. Then there is a natural equivalence
		\begin{equation*}
			(E^{\mathrm{triv}}\otimes F)_{\mathds{h}G}\simeq E\otimes F_{\mathds{h}G}.
		\end{equation*}
		\item[2)] The functor $\mathbb{D}_G:\SH{S}\to\SH{S}$ is op-lax symmetric monoidal.
		\item[3)] \label{prop: extendedpowersandorientedspectra}
		Let $A\in\mathrm{NAlg}(\SH{S})_{\MGL/}$ and let $E\in\mathrm{Mod}_A(\SH{S})$. There is a canonical equivalence
		\begin{equation*}
			\mathbb{D}_G(\Sigma^{2,1}E)\simeq \Sigma^{2n,n}\mathbb{D}_G(E).
		\end{equation*}
	\end{itemize}
\end{proposition}
\begin{proof}
	These are \cite[Lemma C.6, Lemma 6.11 and 6.12, and Proposition 5.37]{MotivicColimits} respectively. Notice that the proof of 3) given in \cite{MotivicColimits} goes through unchanged for subgroups $G\subseteq \Sigma_n$, using \cite[Definition 5.6 (3)]{MotivicColimits} to see that for all $S$-schemes $f:X\to S$ and all $G$-principal bundles $p:U\to X\in\Bet{\Sigma_n}$, the Weil restriction $p_\ast O_U$ is still given by $\A^n$. See also the proof of \cite[Proposition 6.8]{MotivicColimits}.
\end{proof}
\begin{example}
	For $G=C_p\subseteq \Sigma_p$ we in particular obtain from \Cref{prop: extendedpowersandorientedspectra} that
	\begin{equation*}
		\mathbb{D}_{C_p}(\Sigma^{2,1}E)\simeq \Sigma^{2n,n}\mathbb{D}_{C_p}(E).
	\end{equation*}
\end{example}
\begin{remark}
	\label{rmk: normed pointed motivic spaces}
	The above constructions work verbatim for pointed motivic spaces $\Spc{-}_\bullet$, which is a normed category by \cite[Section 6.1]{NormsBachmannHoyois}, instead of the motivic stable homotopy category.
\end{remark}
We also need the following construction, see \cite[Section 6.2]{MotivicColimits}.
\begin{construction}
	\label{constr: Diagonal}
	Fix a natural number $n\in\mathbb{N}$ and let $S$ be a scheme. Let $X\in\Sm_{S}$. Consider $X$ with the trivial $\Sigma_n$-action and $X^{\times n}$ with the $\Sigma_n$-action given by permuting the factors. We obtain a natural $\Sigma_n$-equivariant structure on the diagonal $\Delta_n:X\to X^{\times n}$, natural in $X$, inducing a natural transformation
	\begin{equation*}
		(\Delta_n:\mathrm{triv}\Rightarrow (-)^{\times n}):(\Sm_{S})_+\to(\Sm_{S}^{\Sigma_n})_+.
	\end{equation*}
	The diagonal preserves Nisnevich and motivic equivalences on presheaves, so the above natural transformation induces a natural transformation
	\begin{equation*}
		(\Delta_n:\mathrm{triv}\Rightarrow (-)^{\times n}):\Spc{S}_\bullet\to\Spc{S}_\bullet^{\Sigma_n}.
	\end{equation*}
\end{construction}

\subsection{Definition of the power operations}
\label{Sec: definitionmotivicpoweroperations}

The central technical aspect needed in our proof is the construction of the power operations. 
\begin{definition}
	Let $S$ be a scheme, let $X\in\mathcal{P}(\Sm_{S})$ be a presheaf over $S$, and let $\MGL\to E\in\mathrm{NAlg}(\SH{S})$ be an oriented normed ring spectrum. We define the \emph{motivic $C_p$-power operation} as the composition
	\begin{align*}
		P_{C_p}:E^{2\ast,\ast}(X)=[\sus\LM X,\Sigma^{2\ast,\ast}E]\xrightarrow{\mathbb{D}_{C_p}(-)}&[\mathbb{D}_{C_p}(\sus\LM X),\mathbb{D}_{C_p}(\Sigma^{2\ast,\ast}E)]\\
		\cong& [\mathbb{D}_{C_p}(\sus\LM X),\Sigma^{2\ast p,\ast p}\mathbb{D}_{C_p}(E)]\\
		\xrightarrow{\Delta_p^\ast}& E^{2\ast p,\ast p}(X\times \Bet{C_p})
	\end{align*}
	where the equivalence is given by \Cref{prop: extendedpowersandorientedspectra}, and the last morphism is the diagonal natural transformation \Cref{constr: Diagonal}.
\end{definition}

\begin{remark}
	For a cohomology class $f\in E^{2n,n}(X)$, represented by a morphism $f:\Sigma^\infty_+X\to \Sigma^{2n,n}E$, consider the external $n$-th power of $f$
	\begin{equation*}
		(\Sigma^\infty_+X)^{\otimes p}\xrightarrow{f^{\otimes p}}(\Sigma^{2n,n}E)^{\otimes p}\xrightarrow{\mu}E
	\end{equation*}
	where the last morphism is the multiplication on $E$.
	
	Using \cite[Remark 6.14]{MotivicColimits}, we obtain a commutative diagram
	\[\begin{tikzcd}
		{(\Sigma^\infty_+ X)^{\otimes p}} & {(\Sigma^{2n,n}E)^{\otimes p}} & {\Sigma^{2np,np} E^{\otimes p}} & \Sigma^{2np,np}E \\
		{\mathbb{D}_{C_p}(\Sigma^\infty_+ X)} & {\mathbb{D}_{C_p}(\Sigma^{2n,n}E)} & {\Sigma^{2np,np}\mathbb{D}_{C_p}(E)}
		\arrow[from=1-1, to=1-2]
		\arrow[from=1-1, to=2-1]
		\arrow["\simeq", from=1-2, to=1-3]
		\arrow[from=1-2, to=2-2]
		\arrow["\mu", from=1-3, to=1-4]
		\arrow[from=1-3, to=2-3]
		\arrow[from=2-1, to=2-2]
		\arrow["\simeq", from=2-2, to=2-3]
		\arrow[from=2-3, to=1-4]
	\end{tikzcd}\]
	Here the last diagonal morphism is the morphism induced by the free forgetful adjunction of normed motivic spectra, see \cite[Section 16.4]{NormsBachmannHoyois} and \cite[Construction 2.5]{MotivicNormedSpectra}. We call the composition along the lower line of the diagram the \emph{external motivic $C_p$-power operation of $f$} and denote it by $P_{\mathrm{ext}}^{C_p}(f)\in\MGL^{2np,np}(\mathbb{D}_{C_p}(X))$. Then, for $f\in E^{2n,n}(X)$ a cohomology class, the motivic $C_p$-power operation $P_{C_p}(f)$ is represented by the composition 
	\begin{equation*}
		(\Sigma^\infty_+ X)\otimes(\Sigma^\infty_+\Bet{C_p})\to  \mathbb{D}_{C_p}(\Sigma^\infty_+ X)\xrightarrow{P_{\mathrm{ext}}^{C_p}(f)}\Sigma^{2np,np} E,
	\end{equation*}
	the first morphism is induced by the diagonal natural transformation from \Cref{constr: Diagonal} and applying the geometric orbits. 
\end{remark}
\begin{remark}
	By construction, see also \cite[Section 6.1]{MotivicColimits}, we see that pulling back along the inclusion $\iota:X\to X\times\Bet{C_p}$ gives an identification
	\begin{equation*}
		\iota^\ast\circ P_{C_p}=(-)^p:E^{2\ast,\ast}(X)\to E^{2\ast p,\ast p}(X).
	\end{equation*}
\end{remark}
\begin{remark}\label{rmk: multiplicativity PG}
	Let $G$ be a subgroup of $\Sigma_p$. Another consequence of the construction is the multiplicativity of the power operation. Consider the following commutative diagram, expressing the equality $(fg)^p=f^pg^p$
	\[\begin{tikzcd}
		{\Sigma^\infty_+X\otimes\Sigma^\infty_+X} & {(\Sigma^\infty_+X\otimes\Sigma^\infty_+X)^{\otimes p}} && {(\Sigma^aE\otimes\Sigma^b E)^{\otimes p}} && {(\Sigma^{a+b}E)^{\otimes p}} \\
		{\Sigma^\infty_+X} &&&&& {\Sigma^{(a+b)p}E} \\
		{(\Sigma^\infty_+X)^{\otimes p}} & {(\Sigma^\infty_+X)^{\otimes p}\otimes (\Sigma^\infty_+X)^{\otimes p}} && {(\Sigma^a E)^{\otimes p}\otimes (\Sigma^b E)^{\otimes p}} && {\Sigma^{ap}E\otimes \Sigma^{bp}E}
		\arrow[from=1-1, to=1-2]
		\arrow["{{{(f\otimes g)^{\otimes p}}}}", from=1-2, to=1-4]
		\arrow["\simeq"{description}, from=1-2, to=3-2]
		\arrow["{{{(\mu_2\otimes\mu_2)^{\otimes p}}}}", from=1-4, to=1-6]
		\arrow["\simeq"{description}, from=1-4, to=3-4]
		\arrow["{{{\mu_p}}}", from=1-6, to=2-6]
		\arrow[from=2-1, to=1-1]
		\arrow[from=2-1, to=3-1]
		\arrow[from=3-1, to=3-2]
		\arrow["{{{f^{\otimes p}\otimes g^{\otimes p}}}}"', from=3-2, to=3-4]
		\arrow["{{{\mu_p}}}"', from=3-4, to=3-6]
		\arrow["{{{\mu_2}}}"', from=3-6, to=2-6]
	\end{tikzcd}\]
	Here we use that $E$ is normed, hence commutative, for the commutativity of the right pentagon.
	
	Applying extended powers to this diagram and using the the oplax monoidality of $\mathbb{D}_{G}$, see \cite[Proposition 6.12]{MotivicColimits}, we obtain a commutative diagram
	\[\begin{tikzcd}
		& {\mathbb{D}_G(\Sigma^\infty_+X\otimes\Sigma^\infty_+X)} && {\mathbb{D}_G(\Sigma^aE\otimes\Sigma^b E)} && {\mathbb{D}_G(\Sigma^{a+b}E)} \\
		{\Sigma^\infty_+X\otimes\Sigma^\infty_+\Bet{C_p}} &&&&& {\Sigma^{(a+b)p}E} \\
		& {\mathbb{D}_G(\Sigma^\infty_+X)\otimes \mathbb{D}_G(\Sigma^\infty_+X)} && {\mathbb{D}_G(\Sigma^a E)\otimes \mathbb{D}_G(\Sigma^b E)} && {\Sigma^{ap}E\otimes \Sigma^{bp}E}
		\arrow["{{{\mathbb{D}_G(f\otimes g)}}}", from=1-2, to=1-4]
		\arrow[from=1-2, to=3-2]
		\arrow["{{\mathbb{D}_G(\mu_2\otimes\mu_2)}}", from=1-4, to=1-6]
		\arrow[from=1-4, to=3-4]
		\arrow["{{{\mathbb{D}_G\mu_p}}}", from=1-6, to=2-6]
		\arrow[from=2-1, to=1-2]
		\arrow[from=2-1, to=3-2]
		\arrow["{{{\mathbb{D}_Gf\otimes \mathbb{D}_Gg}}}"', from=3-2, to=3-4]
		\arrow["{{\mathbb{D}_G\mu_p}}"', from=3-4, to=3-6]
		\arrow["{{{\mathbb{D}_G\mu_2}}}"', from=3-6, to=2-6]
	\end{tikzcd}\]
	The upper composition represents the object $P_{G}(fg)$, while the lower represents $P_{G}(f)P_{G}(g)$. Commutativity of the diagram then shows that the two expression are equal, that is $P_{G}(fg)=P_{C_p}(f)P_{C_p}(g)$, that is, $P_{G}$ is multiplicative.
\end{remark}

\subsection{The normed slice category}
We also need the technical construction of the slice category over Weibel's homotopy $K$-theory space $\Spc{-}_{/K^{\mathrm{hom}}}$ as a normed category.
\begin{construction}
	By \cite[Proposition 15.26]{NormsBachmannHoyois} we know that the $K$-theory presheaf $K:\mathrm{Sch}^{\mathrm{op}}\to\mathrm{CAlg}(\An)$ can be extended to a functor
	\begin{equation*}
		K:\Span(\mathrm{Sch},\mathrm{all},\mathrm{f\Acute{e}t})\to \mathrm{CAlg}(\An).
	\end{equation*}
	By \cite[Remark 6.3]{NormsBachmannHoyois} $\Spc{-}$, is a normed category. Let $p:\int \Spc{-}\to \Span(\mathrm{Sch},\mathrm{all},\mathrm{f\acute{e}t})$ be the cocartesian fibration classifying the functor $\Spc{-}$ and let $K^{\mathrm{hom}}:=\Omega^\infty KH\simeq L_{\mathrm{cdh}}K$ be the homotopy $K$-theory space, where the equivalence is \cite[Remark 16.11]{NormsBachmannHoyois}. By \cite[Theorem 15.22 and Corollary 15.23]{NormsBachmannHoyois}, $K^{\mathrm{hom}}$ is a section of $p$. We obtain the slice category $\Spc{-}_{/K^{\mathrm{hom}}}$ as the pullback
	\[\begin{tikzcd}
		{\int\Spc{-}_{/K^{\mathrm{hom}}}} & {\mathrm{Fun}(\Delta^1,\int\Spc{-})} \\
		{\Span(\mathrm{Sch},\mathrm{all},\mathrm{f\acute{e}t})} & {\int\Spc{-}}.
		\arrow[from=1-1, to=1-2]
		\arrow["q"', from=1-1, to=2-1]
		\arrow["\lrcorner"{anchor=center, pos=0.125}, draw=none, from=1-1, to=2-2]
		\arrow["\mathrm{ev_1}", from=1-2, to=2-2]
		\arrow["K^{\mathrm{hom}}", from=2-1, to=2-2]
	\end{tikzcd}\]
	As $\mathrm{ev}_1$ is cocartesian, see \cite[Theorem 5.9.3]{cisformal}, so is its pullback $q$. Hence the straightening $\Spc{-}_{/K^{\mathrm{hom}}}:\Span(\mathrm{Sch},\mathrm{all},\mathrm{f\acute{e}t})\to \Cat$ of $\int\Spc{-}_{/K^{\mathrm{hom}}}$ is a normed category, see also \cite[Theorem 2.2.2.4]{lurieha} for this procedure for general slice operads. By \cite[Remark 2.2.2.5]{lurieha} this is sectionwise the slice category $S\mapsto \Spc{S}_{/K^{\mathrm{hom}}|_S}$.     
\end{construction}

\begin{remark}
	To see that the straightening $\int\Spc{-}_{/K^{\mathrm{hom}}}$ is section wise the slice category, at least heuristically on objects, consider the description of objects in a pullback of $\infty$-categories. In our case, an object in $\int\Spc{-}_{/K^{\mathrm{hom}}}$ is a scheme $S$ together with a morphism in $\int\Spc{-}$, such that the target of the morphism is equivalent to the object $(S,K^{\mathrm{hom}}|_S)$. Phrased differently, an object in $\int\Spc{-}_{/K^{\mathrm{hom}}}$ consists of a scheme $S$ and a morphism $X\to K^{\mathrm{hom}}|_S\in \Spc{S}$, where $X\in\Spc(S)$ is a pointed motivic space over $S$. Thus, $\Spc{S}_{/K^{\mathrm{hom}}|_S}$, which is the fiber of $q$ over $S$, has the desired description. 
\end{remark}

\subsection{Motivic Landweber-Novikov operations}

Recall that an orientation of a homotopy commutative ring spectrum $E\in\CAlg(h\SH{S})$ is given by a map of homotopy commutative rings $\MGL\to E\in\CAlg(h\SH{S})$. It is well known, that such an orientation unlocks a theory of Chern classes for $E$, see for example \cite{deglise2024characteristic}. Using the theory of Chern classes, together with the Thom isomorphism \cite[Proposition 16.28]{NormsBachmannHoyois}, we can define a motivic analog of the so-called Landweber-Novikov operations, see \cite[Section I.4 and I.5]{adams1974stable}, .

To define the Landweber-Novikov operations, let $E$ be an oriented ring spectrum with orientation defined by a morphism $\varphi:\MGL\to E\in\CAlg(h\SH{S})$. By \cite[Proposition 6.2]{naumann2009motiviclandweberexactness} we obtain an equivalence
\begin{align*}
	E^{\ast,\ast}(\Bet{GL}_d)&\cong E^{\ast,\ast}[[c_1,\ldots, c_d]],\\
	E^{\ast,\ast}(\Bet{GL})&\cong E^{\ast,\ast}[[c_1,\ldots]]
\end{align*}
where $\Bet{\mathrm{GL}}=\colim{d}{\Bet{\mathrm{GL}_d}}$ and the $c_i$ are the universal $E$-Chern classes. If $\varphi=\mathrm{id}:\MGL\to\MGL$, we call the $c_i$ the \emph{motivic Conner-Floyd-Chern classes}.

Just as in classical topology, one can furthermore show that
\begin{equation*}
	E^{\ast,\ast}(\Bet{GL})\hookrightarrow (E^{\ast,\ast}(\Bet{\mathbb{G}_m}))^{\times\infty}\simeq E^{\ast,\ast}[[x_1,x_s,\ldots]],
\end{equation*}
where the image can be identified with the subring of symmetric polynomials.
\begin{definition}
	\label{defn: Landwebernovikovoperationsmotivically}
	Let $\alpha=(\alpha_1,\alpha_2,\ldots)$ be a non-negative multiindex with only finitely many non-zero entries. Let $c_\alpha\in E^{\ast,\ast}[[x_1,x_2,\ldots]]$ denote the symmetric power series whose monomials have exactly $\alpha_j$ many variables of exponent $j$. We define the \emph{$\alpha$ Landweber-Novikov operation with respect to $\varphi$}, $s_\alpha\in E^{\ast,\ast}\MGL$ as the image of $c_\alpha\in E^{\ast,\ast}(\Bet{GL})$ under the $E^{\ast,\ast}(X)$-module Thom isomorphism.
	
	If $\varphi = \mathrm{id}$, we call the $s_\alpha$ simply the Landweber-Novikov operations.
\end{definition}
\begin{example}
	We can recover the motivic Conner-Floyd-Chern classes from the $c_\alpha$ by taking the multiindices $\alpha=(n,0,0,\ldots)$. Indeed, in this case, $c_\alpha$ is the power series
	\begin{equation*}
		c_{(n,0,0,\ldots)}=\underset{|S|=n}{\sum_{S\subseteq \N}} \prod_{j\in S}x_j.
	\end{equation*}
	As another example, consider $\alpha=(1,0,2,0,0\ldots)$. Then
	\begin{equation*}
		c_\alpha = \sum_{i}\sum_{j\neq i}\underset{k>j}{\sum_{k\neq i}}x_ix_j^3x_k^3.
	\end{equation*}
	See also \cite[Definition 2.5.1]{FGBO}.
\end{example}
As the Landweber-Novikov operations are just the images of Chern classes under a module homomorphism, all the properties of Chern classes translate into those of the Landweber Novikov operations.
\begin{proposition}
	Let $\varphi=\mathrm{id}:\MGL\to\MGL$. Let $\xi,\eta$ be two rank $n$ vector bundles over a scheme $X$, and let $x,y\in\MGL^{\ast,\ast}(X)$ be two cohomology classes. Then the motivic Conner-Floyd-Chern classes and the motivic Landweber-Novikov operations satisfy the following properties:
	\begin{itemize}
		\item[1)] $c_0(\xi)=1\in\MGL^{0,0}(X)$, and $s_0=\mathrm{id}$ is the identity.
		\item[2)] Let $g:Y\to X\in\mathrm{Sch}$ be a morphism. Then $g^\ast c_\alpha(\xi)=c_\alpha(g^{\ast}\xi)$ and $g^\ast s_\alpha = s_\alpha g^\ast$.
		\item[3)] The Landweber-Novikov operations are stable and hence additive.
		\item[4)] Define the total Chern class
		\begin{equation*}
			c_{\mathbf{t}}(\xi):=\sum_\alpha c_\alpha(\xi) t_1^{\alpha_1}t_2^{\alpha_2}\ldots \in\MGL^{\ast,\ast}(X)[[t_1,t_2,\ldots]]
		\end{equation*}
		and the total Landweber-Novikov operation
		\begin{equation*}
			s_{\mathbf{t}}(x):=\sum_\alpha s_\alpha(x) t_1^{\alpha_1}t_2^{\alpha_2}\ldots \in\MGL^{\ast,\ast}(X)[[t_1,t_2,\ldots]],
		\end{equation*}
		where $\alpha$ ranges over all non-negative multi-indices with only finitely many non-zero entries.
		
		Then the Cartan formula holds:
		\begin{align*}
			c_{\mathbf{t}}(\xi\oplus\eta)&=c_{\mathbf{t}}(\xi)c_{\mathbf{t}}(\eta),\\
			s_{\mathbf{t}}(xy)&=s_{\mathbf{t}}(x)s_{\mathbf{t}}(y).
		\end{align*}
		\item[5)] The Conner-Floyd-Chern classes and Landweber-Novikov operations are related by the Thom isomorphism $\varphi:\MGL^{\ast,\ast}(X)\xrightarrow{\cong} \MGL^{\ast,\ast}(\mathrm{Th}_X(\xi))$
		\begin{equation*}
			\varphi(c_\alpha(\xi))=s_\alpha(\varphi(1)).
		\end{equation*}
	\end{itemize}
\end{proposition}
\begin{proof}
	Using the properties of Chern-classes from \cite{degliseorientation}, the proof is verbatim the same as in topology, see \cite[Section I.4 and I.5]{adams1974stable}.
\end{proof}

\subsection{\texorpdfstring{Properties of $\MGL$ and oriented spectra}{Properties of MGL and oriented spectra}}
We end the section with a short summary of properties of $\MGL$.
\begin{theorem}
	\label{MotCohomDiag}
	There is an isomorphism of graded rings
	\begin{equation}
		\underline{\pi}_{0}(\HZ)_\ast(k)\cong H^{-\ast,-\ast}\left(k,\mathbb{Z}\right)\cong K^{\mathrm{M}}_{-\ast}\!(k).
	\end{equation}
	If $l\in\mathbb{N}$ there is an isomorphism of graded rings
	\begin{equation}
		\underline{\pi}_{0}(\HZ/l)_\ast(k)\cong H^{-\ast,-\ast}\left(k,\mathbb{Z}/l\right)\cong K^{\mathrm{M}}_{-\ast}\!(k)/l.
	\end{equation}
\end{theorem}
\begin{proof}
	The first claim is \cite[Theorem 5.1]{mazweibvoe} while the second claim follows from the long exact sequence on global sections of the homotopy sheaves induced by the exact sequence $\HZ\xrightarrow{l}\HZ\to\HZ/l$ together with the fact that $\HZ$ is connective, see \cite[Lemma 7.3]{hoyois2017six}. We furthermore used that every field is in particular a henselian local ring, hence taking global sections is the same as talking stalks at $k$, and that by \cite[\href{https://stacks.math.columbia.edu/tag/04EN}{Tag 04EN}]{stacks-project} the stalk functor is exact.
\end{proof}
For a positive integer $l$, we denote by $\mu_l$ the sheaf of $l$-th roots of unity.
\begin{theorem}[Norm-Residue-Isomoprhism]
	\label{Thm: Norm-Residue-Iso}
	Let $l\in \mathbb{N}$ be coprime to the exponential characteristic of $k$. There is a natural isomorphism
	\begin{equation}
		H^{p,q}\left(k,\mathbb{Z}/l\right)\cong H_{\mathrm{\text{ét}}}^p(k,\mu_l^{\otimes q})
	\end{equation}
	for $p\leq q$.
\end{theorem}
\begin{proof}
	This is a deep theorem by Voevodsky, see \cite[Theorem 6.1]{voevodsky2011motivic}.
\end{proof}
We obtain an immediate corollary:
\begin{corollary}
	\label{Norm-residue-isomorphism}
	Let $k$ be a field containing a primitive $l$-th root of unity. There is a non-canonical isomorphism of bigraded rings\footnote{Leaving out the bigradings for notational simplicity.}
	\begin{equation*}
		 \left(K^{\mathrm{M}}_{\ast}\!(k)/l\right)\left[\tau\right]\xrightarrow{\cong}\underline{\pi}_{\ast}(\HZ/l)_{\ast}(k)\cong H^{\ast,\ast}\left(k,\mathbb{Z}/l\right)
	\end{equation*}
	where $\tau$ is mapped to a choice $\zeta$ of a primitive $l$-th root of unity in $H^{0,1}\left(k,\mathbb{Z}/l\right)\simeq H_{\mathrm{\text{ét}}}^0(k,\mu_l)=\mu_l(k)$ .
\end{corollary}
\begin{proof}\
	After fixing a choice of a $l$-th root of unity, the second isomorphism of \Cref{MotCohomDiag} induces a canonical morphism of bigraded rings
	\begin{equation*}
		\left(K^{\mathrm{M}}_{\ast}\!(k)/l\right)\left[\tau\right]\to H^{\ast,\ast}\left(k,\mathbb{Z}/l\right).
	\end{equation*}
	Hence it suffices to check that this induces an isomorphism degreewise, that is, that\footnote{The left hand side denotes the $(p,q)$-graded part of Milnor $K$-theory adjoined variable $\tau$.} $((K^{\mathrm{M}}_{\ast}\!(k)/l)\left[\tau\right])^{(p,q)}\xrightarrow{\cong} H^{p,q}\left(k,\mathbb{Z}/l\right)$ for all $p,q\in\mathbb{Z}$. Both sides vanish for $p>q$, and for $p=q$ this is true by construction. The claim now follows from \Cref{Thm: Norm-Residue-Iso}. Indeed, as $k$ is a field, there are isomorphisms
	\begin{equation*}
		H_{\mathrm{\text{ét}}}^p(k,\mu_l^{\otimes q})\cong H_{\mathrm{Gal}}^p(\bar{k}/k,\mu_l^{\otimes q})\xleftarrow{\cong}H_{\mathrm{Gal}}^p(\bar{k}/k,\mu_l).
	\end{equation*}
	Here the second isomorphism is given by multiplication with an primitive $l$-th root of unity, that is, the image of $\tau$ see also \cite[(7.1)]{KyllingRondigsOstvaer2020}.
\end{proof}
The norm-residue theorem has the following corollary
\begin{corollary}
	\label{cor: homotopygroupsofZpcompleted}
	Let $k$ be a field containing a compatible system of primitive $p$-th power roots of unity. Then there exists a isomorphism of bigraded rings
	\begin{equation}
		\label{Normresidue2}
		\pi_{\ast}\left(\HZp\right)_{\ast}\cong \KM{\ast}{k}^\wedge_p\left[\tau\right].
	\end{equation}
	Here $\tau\in \pi_{1}\left(\HZp\right)_1$ is a so-called \emph{Tate twist element}.
\end{corollary}
\begin{proof}
	The isomorphism is induced by the Milnor sequence
	\[\begin{tikzcd}
		0 & {\underset{n}{\mathrm{lim}^1}\pi_{\ast+1}(H\mathbb{Z}/p^n)_\ast} & {\pi_{\ast}(\HZp)_\ast} & {\ilim{n}   {\pi_{\ast}(H\mathbb{Z}/p^n)_\ast}} & 0
		\arrow[from=1-1, to=1-2]
		\arrow[from=1-2, to=1-3]
		\arrow[from=1-3, to=1-4]
		\arrow[from=1-4, to=1-5]
	\end{tikzcd}\]
	since the $\mathrm{lim}^1$-term vanishes. Indeed, by the norm-residue theorem, there is an isomorphism of towers 
	\begin{equation*}
		\left\{\pi_{\ast}(H\mathbb{Z}/p^n)_\ast\right\}_{n>0}\cong\left\{(\KM{\ast}{k}/p^n)\left[\tau\right]\right\}_{n>0}.    
	\end{equation*}
	The transition maps on the right hand side are the canonical projection maps. In particular, they are surjective. So the Mittag-Leffler condition is satisfied and we obtain the desired isomorphism.
\end{proof}

\begin{remark}
	\label{rmk: tattwistinHZ}
	Notice that there is no canonical Tate twist. Indeed, it follows from the proof that $\tau$ corresponds to a primitive $p^n$-th root of unity $\tau_n$ along the composition
	\begin{equation*}
		\pi_{\ast}\left(\HZ^\wedge_p\right)_{\ast}\to\pi_{\ast}\left(\HZ/p^n\right)_{\ast}\simeq \KM{\ast}{k}/p^n[\tau_n].
	\end{equation*}
	The set $\{\tau_n\}_{n\geq 0}$ furthermore is compatible, so it corresponds to a compatible system of primitive $p$-th power roots of unity. For more discussion regarding the Tate twist element, see \cite[Remark under Lemma 23]{hu11}, see \cite[Remark beneath Lemma 23]{gh21}, or \cite{gh18} for the explicit construction in the $p=2$ case.
	
\end{remark}

\begin{proposition}
	\label{prop: MGLandHZsame}
	There is an equivalence of ring spectra
	\begin{equation*}
		\MGL_{\leq 0}\simeq \HZ_{\leq 0}
	\end{equation*}
\end{proposition}
\begin{proof}
	This is \cite[Lemma 7.5]{hoy15}.
\end{proof}
We now specialize to the case where our base scheme is the spectrum of a field $k$. In this case we obtain the following.
\begin{corollary}
	\label{Cor: MGLMilnorKthy}
	There is a natural isomorphism of graded rings
	\begin{equation}
		\label{MGLnormresidue}
		\pi_{0}\left(\MGL(k)\right)_{\ast}\cong \KM{-\ast}{k}.
	\end{equation}
	Furthermore, $\pi_{k}\left(\MGL(k)\right)_{\ast}=0$ for all $k<0$. Here we write $\MGL(k):=\MGL(\mathrm{Spec}(k))$ for brevity.
\end{corollary}
\begin{proof}
	This is a direct consequence of \Cref{prop: MGLandHZsame} and the isomorphism \eqref{MotCohomDiag}
\end{proof}
In particular, we obtain the following description of the diagonal of the $p$-completion of $\MGL$:
\begin{corollary}
	There is a natural isomorphism of graded rings
	\begin{equation}
		\label{MGLMilnorK}
		\pi_{0}((\MGL(\mathrm{Spec}(k)))^\wedge_p)_{\ast}\cong\KM{\ast}{k}^\wedge_p.
	\end{equation}
\end{corollary}
\begin{proof}
	As $\SH{-}$ is stable, $p$-completion commutes with global sections, that is
	\begin{align*}
		\MGL^\wedge_p(k)=\MGL(k)^\wedge_p.
	\end{align*}
	Hence, we can simply write $\MGL^\wedge_p:=\MGL^\wedge_p(k)$ without any ambiguity, and similarly $\HZ^\wedge_p:=\HZ^\wedge_p(k)$. We further abbreviate $\MGL:=\MGL(k)$.
	
	Consider the Milnor exact sequence, see \cite[Theorem 2.33]{nikolaus2016rational}
	\[\begin{tikzcd}
		0 & {\mathbb{L}_0\pi_{0}(\MGL)_{n}} & {\pi_{0}(\MGL^\wedge_p)_{n}} & {\mathbb{L}_1\pi_{-1}(\MGL)_n} & 0,
		\arrow[from=1-1, to=1-2]
		\arrow[from=1-2, to=1-3]
		\arrow[from=1-3, to=1-4]
		\arrow[from=1-4, to=1-5]
	\end{tikzcd}\]
	where $\mathbb{L}_0,\mathbb{L}_1$ are the zeroth and first derived completions.
	
	As $\MGL$ is connective by \Cref{prop: MGLandHZsame}, the right most term vanishes and we obtain an isomorphism of graded rings
	\begin{equation*}
		\pi_{0}(\MGL^\wedge_p)_\ast\cong \mathbb{L}_0\pi_{0}(\MGL)_\ast\cong \mathbb{L}_0\KM{\ast}{k}.
	\end{equation*}
	The Bousfield-Kan sequence for $\HZ^\wedge_p$ on the other hand gives an isomorphism
	\begin{equation*}
		\KM{\ast}{k}^\wedge_p\cong \pi_{0}(\HZ^\wedge_p)_\ast\cong \mathbb{L}_0\pi_{0}(\HZ)_\ast\cong \mathbb{L}_0\KM{\ast}{k}.
	\end{equation*}
	Thus, the zeroth derived $p$-completion of Milnor $K$-theory agrees with its the classical $p$-completion, which implies the claim.
\end{proof}
\begin{proposition}
	\label{prop: tatetwistinmgl}
	Let $k$ be a field containing a fixed compatible system of all primitive $p$-th power roots of unity. Then there exists an element $\tau\in\pi_1(\MGL^\wedge_p(k))_1$, such that, under the canonical morphism $\pi_\ast(\MGL^\wedge_p(k))_{\ast}\to\pi_\ast(\HZ^\wedge_p(k))_\ast$ induced by the orientation, the element $\tau$ is mapped to the Tate twist from \Cref{cor: homotopygroupsofZpcompleted}.
\end{proposition}
\begin{proof}
	Let $E$ by a spectrum and let $m\in\Z$ be an integer. Denote by $m:E\to E$ the multiplication-by-$m$ map, that is the composition $E\to E^{\oplus m}\to E$, where the first map is the diagonal, while the second map is the codiagonal. Naturality of this construction in the divisors of $m$ gives us, that for all $n>1$ there exists a commutative diagram in $\SH{k}$, where the rows are cofiber sequences
	\[\begin{tikzcd}
		{\MGL} & {\MGL} & {\MGL/p^n} \\
		{\MGL} & {\MGL} & {\MGL/p^{n-1}}
		\arrow["p^n", from=1-1, to=1-2]
		\arrow["p"', from=1-1, to=2-1]
		\arrow[from=1-2, to=1-3]
		\arrow["{\mathrm{id}}"', from=1-2, to=2-2]
		\arrow[from=1-3, to=2-3]
		\arrow["{p^{n-1}}", from=2-1, to=2-2]
		\arrow[from=2-2, to=2-3].
	\end{tikzcd}\]
	Here the right-most vertical arrow is the canonical projection $\MGL/p^n\to\MGL/p^{n-1}$, which is just the map induced on cofibers by the commutativity of the left hand square.
	
	Taking global sections\footnote{We continue with the convention from the last proof and simply write $\MGL$, $\MGL/p^n$, and $\MGL^\wedge_p$ for the global sections, which is justified by the equivalence $\MGL/p^n(k)\simeq\MGL(k)/p^n$, using $E(k)\simeq \mathrm{map}(\Sigma^\infty_+\Spec(k),E)$ and the stability of $\SH{k}$.}, or equivalently, stalks, and applying homotopy groups, we obtain a commutative diagram of homotopy groups with exact rows
	\[\begin{tikzcd}
		{\pi_{1}(\MGL)_1(k)} & {\pi_{1}(\MGL/p^{n})_1} & {\KM{1}{k}} & {\KM{1}{k}} \\
		{\pi_1(\MGL(k))_1} & {\pi_{1}(\MGL/p^{n-1}(k))_1(k)} & {\KM{1}{k}} & {\KM{1}{k}}
		\arrow[from=1-1, to=1-2]
		\arrow["{\mathrm{id}}", from=1-1, to=2-1]
		\arrow[from=1-2, to=1-3]
		\arrow[from=1-2, to=2-2]
		\arrow["{p^{n}}", from=1-3, to=1-4]
		\arrow["p", from=1-3, to=2-3]
		\arrow["{\mathrm{id}}", from=1-4, to=2-4]
		\arrow[from=2-1, to=2-2]
		\arrow[from=2-2, to=2-3]
		\arrow["{p^{n-1}}", from=2-3, to=2-4]
	\end{tikzcd}\]
	Here we used the isomorphism $\pi_{0}(\MGL(k))_1\cong\KM{1}{k}$ of \Cref{Cor: MGLMilnorKthy}.
	
	We inductively construct elements $\tau_n\in\pi_{1}\left(\MGL/p^n(k)\right)_1$ such that they are compatible under the maps induced by the canonical maps $\MGL/p^n\to\MGL/p^{n-1}$. The construction is essentially the same as how one proves the surjectivity of the five lemma, so the reader how does not want to concern himself with the details can safely skip this part and hence the rest of the proof.
	
	For $n=1$ take any preimage of $\left[\zeta_p\right]\in\KM{1}{k}$. Assume now for $n>1$ that $\tau_m$ has been constructed for all $m<n$. Take any preimage $\Tilde{\tau}_n\in\pi_{1}(\MGL/p^{n})_1$ of $\left[\zeta_{p^{n}}\right]\in\KM{1}{k}^\wedge_p$. As the $p$-power roots of unity were chosen compatible, the image $\Tilde{\tau}_{n-1}\in\pi_{1}(\MGL/p^{n-1})_1$ of $\Tilde{\tau}_n$ maps to $\left[\zeta_{p^{n-1}}\right]\in\KM{1}{k}^\wedge_p$. In particular, the difference $\tau_{n-1}-\Tilde{\tau}_{n-1}$ maps to zero in $\KM{1}{k}^\wedge_p$. Thereby, there exists an element $\beta\in \pi_{1}(\MGL_p^{\wedge})_1$ mapping to $\tau_{n-1}-\Tilde{\tau}_{n-1}$. If we denote by $b$ the image of $\beta$ in ${\pi_{1}(\MGL/p^{n})_1}$, then $\tau_n:=\Tilde{\tau}_n+b$ maps to $\left[\zeta_{p^{n}}\right]$ in $\KM{1}{k}^\wedge_p$, and to $\tau_{n-1}$ in $\pi_{1}(\MGL/p^{n-1})_1$. This defines the compatible sequence $\left\{\tau_n\right\}_{n>0}$, and hence the desired element $\tau\in\pi_{1}(\MGL^\wedge_p)_1$.
	
	Notice that the canonical morphism induces an equivalence on $\pi_0$ by \Cref{prop: MGLandHZsame}. Indeed, the canonical morphism is a ring homomorphism, hence is $\pi_{\ast}(\mathbb{S})_{\ast}$-linear, and in particular the identity on Milnor $K$-theory. Consider the following diagram:
	\[\begin{tikzcd}
		{\pi_{1}(\MGL/p^n)_{1}} & {\pi_{0}(\MGL)_{1}} & {\pi_{0}(\MGL)_{-1}} \\
		& {\KM{1}{k}} & {\KM{1}{k}} \\
		{\pi_{1}(\mathrm{H}\mathbb{Z}/p^n)_{1}} & {\pi_{0}(\HZ)_{1}} & {\pi_{0}(\HZ)_{1}}
		\arrow[from=1-1, to=1-2]
		\arrow[from=1-1, to=3-1]
		\arrow["{{p^n}}", from=1-2, to=1-3]
		\arrow["\cong"', from=1-2, to=2-2]
		\arrow["\cong", from=1-3, to=2-3]
		\arrow["{{p^n}}", from=2-2, to=2-3]
		\arrow["\cong"', from=2-2, to=3-2]
		\arrow["\cong", from=2-3, to=3-3]
		\arrow[from=3-1, to=3-2]
		\arrow["{{p^n}}", from=3-2, to=3-3].
	\end{tikzcd}\]
	By commutativity, the compatible system $\tau_n$ is mapped under the left vertical map to a compatible system of preimages of the $p$-th power roots of unity. Thus, $\tau \in\pi_1(\MGL^\wedge_p)_1$ is mapped to some element in $\pi_1(\HZ)_1$, which itself is mapped to the choosen compatible system of $p$-th power roots of unity. But $\tau\in \pi_{1}(\HZp)_{1}$ is the unique element mapped to this compatible system, see \Cref{rmk: tattwistinHZ}. So the $\tau$ we defined above is mapped to the Tate-twist $\tau$ under the map induced by the standard orientation on $\HZ$. 
\end{proof}
\begin{remark}\label{Remark: Tate twist in MGL}
	Notice that for the last part of the proof of \Cref{prop: tatetwistinmgl}, it was really necessary to fix the compatible system of primitive $p$-th power roots of unity, so that the Tate twist and $\tau\in\pi_1(\MGL^\wedge_p(k))_1$ are constructed using the same system of roots of unity. Else, the image of $\tau$ could differ by a element in $(\Z^\wedge_p)^\times$, see also \Cref{rmk: tattwistinHZ}.
\end{remark}
We end this section with a calculation.
\begin{proposition}
	\label{prop: homotopysheavesofMK}
	Denote by $\KMS{\ast}:=\underline{K}^{\mathrm{M}}_\ast/p$ the mod $p$ reduction of the Milnor $K$-theory sheaf. There is an explicit description of the homotopy sheaves of the effective cover of $\KMS{\ast}$:
	\begin{align*}
		\HSeff{i}{f_0\KMS{\ast}}\simeq\begin{cases}
			\HSeff{0}{\MGL/p}\quad &\text{for }i=0\\
			0 &\text{else}.
		\end{cases}
	\end{align*}
\end{proposition}

\begin{proof}
	Since the functors $\HS{0}{0}{-}:\SH{k}^{\mathrm{eff},\heartsuit}\to\mathrm{Ab}(Shv_{\mathrm{Nis}}(Sm(k)))$ are conservative, see \cite[Proposition 5 1)]{Bachmann_2017}, it suffices to calculate
	\begin{equation*}
		\HS{0}{0}{\iota^\heartsuit\HSeff{i}{f_0\KMS{\ast}}}\simeq \HS{0}{0}{\HS{i}{\ast}{\iota f_0\KMS{\ast}}}\simeq \HS{i}{0}{f_0\KMS{\ast}}.
	\end{equation*}
	From now on we will suppress $\iota$ and $\iota^{\heartsuit}$ in the formulas for clarity.
	
	Using that the suspension spectra are by definition effective, we have
	\begin{align*}
		\HS{i}{0}{f_0\KMS{\ast}}
		&\cong \HS{i}{0}{\KMS{\ast}}
	\end{align*}
	But $\KMS{\ast}\in\SH{k}^{\heartsuit}$, so that these groups vanish for all $i\neq 0$. For $i=0$ we note that there is a map $\MGL/p\to f_0\KMS{\ast}$, induced by $\MGL/p\to \HS{0}{\ast}{\MGL/p}$. Thus, the composition 
	\begin{equation*}
		\HS{0}{0}{\MGL/p}\to\HS{0}{0}{\KMS{\ast}}\cong\HS{0}{0}{\HS{0}{\ast}{\MGL/p}}\cong \HS{0}{0}{\MGL/p}
	\end{equation*}
	is the identity, so that $\HSeff{0}{\MGL/p}\simeq \HSeff{0}{f_0\KMS{\ast}}$, using effectivity of $\MGL$.
\end{proof}
	\section{Quillen's Lemma}
\label{sec: Quillen's lemma}
To calculate the bigraded homotopy groups of $\MGL^\wedge_p$ for some prime $p$, we need control over the elements in these groups. This goal of this section is exactly to develop such a control, using the motivic power operations \Cref{Sec: definitionmotivicpoweroperations}. In fact, we will show that that the motivic power operations for $\MGL$ can be expressed explicitly as a power series.

This whole section follows closely the original proof of Quillen \cite{BordismQuillen}.

\subsection{\texorpdfstring{Calculation of the Power Operations for $\MGL$}{Calculation of the Power Operations for MGL}}

Our precise goal is to calculate the $C_p$-power operations of $\MGL$. We want to reduce this calculation of the $C_p$-power operations to the calculation of a single Chern class. In fact it will be this reduction, besides the definition of the power operations of course, where we need the normed structure of $\MGL$, see \cite[Section 16.3]{NormsBachmannHoyois}.

Let $S$ be a qcqs base scheme, $X\in\Spc{S}^\omega$ a compact motivic space, and $x\in\MGL^{2p,p}(X)$ a cohomology class, which we represent by a morphism $ \Sigma^{-2p,-p}\Sigma_+^{\infty}X\to \MGL$. We define
\begin{equation*}
	\MGL(n)=\Sigma^{-2n,-n}M(\gamma_n),
\end{equation*}
where $\gamma_n$ is the canonical bundle on $\Bet{GL}(n)$, and $M(\gamma_n)=M_S(j\circ\gamma_n)$ is the motivic Thom spectrum of $\gamma_n$, see \cite[Chapter 16]{NormsBachmannHoyois} for the precise definition. By \cite[Theorem 16.13]{NormsBachmannHoyois}, the colimit $\MGL:=\mathrm{colim}\MGL(n)$ carries a normed structure. Furthermore, by \cite[Proposition 6.4 (3)]{hoyois2017six}, the suspension spectrum $\Sigma_+^\infty X$ is compact, hence the morphism representing $x$ factors over a finite stage 
\begin{equation*}
	f:X\to\Omega^\infty\Sigma^{2p,p}\mathrm{MGL}(n)=\Omega^\infty\Sigma^{-2(n+p),-(n+p)}M(\gamma_n)\xrightarrow{\Omega^\infty u_n}\Omega^\infty\Sigma^{2p,p}\MGL    
\end{equation*}
As the power operations are stable under pullback, it suffices to calculate the power operations of the canonical morphisms $u_n:M(\gamma_n)\to \Sigma^{2n,n}\MGL$.

We start with the morphism $u_1:M(\gamma_1)\to \Sigma^{2,1}\MGL$, which is the Thom class of the tautological bundle on the étale classifying stack, see also \cite[Example 16.30]{NormsBachmannHoyois}. Using that $M(\gamma_1)\simeq \Sigma^\infty\Proj^\infty$, see \cite[2.2.5]{degliseorientation}, we identify this morphism with the first Chern class $c_1(\gamma_1)\in \MGL^{2,1}(\mathbb{P}^{\infty})$.


\begin{remark}
	\label{rmk: MGLandPMGLpoweroperations}
	Fix a positive integer $p\in\mathbb{N}$. Before we actually calculate $P_{C_p}(c_1)\in\MGL^{2p,p}(\BGm\times\Bet{C_p})$, we first post-compose the first Chern class with the inclusion $\MGL\to\PMGL$ so that we can consider a morphism $c_1:\Sigma^\infty \BGm\to\PMGL$. Using $\PMGL$ we can get rid of the shifts appearing in the definition of the $\MGL(n)$ which would complicate things unnecessarily later on. By \cite[Theorem 16.19]{NormsBachmannHoyois}, this does not pose a problem. Indeed, given $x\in \MGL^{2n,n}(X)$ and its power operation $P_{C_p}(x)\in\MGL^{2np,np}(X\times \Bet{C}_p)$, then, for $\Tilde{x}\in\PMGL^{0,0}(X)$ and $\Tilde{P}^{C_p}(x)\in\PMGL(X\times \Bet{C_p})$ their respective images under the morphism of cohomology theories induced by the map of normed spectra $\MGL\to\PMGL$, we have $P_{C_p}(\Tilde{x})=\Tilde{P}^{C_p}(x)$. In other words, the morphism of cohomology theories $\MGL^{\ast,\ast}(-)\to\PMGL^{\ast,\ast}(-)$ commutes with power operations. 
\end{remark}

We are now ready to start the computation. Our first goal is to rewrite $P_{C_p}(c_1)$ as the top Chern class of another bundle that is accessible by character theory and the formal group law on $\MGL$.

\begin{proposition}
	We have an equality
	\begin{equation*}
		P_{C_p}(c_1)=c_n(\Delta_{\mathds{h}C_p}^\ast \beta)\in\PMGL^{0,0}(\BGm),
	\end{equation*}
	where $\beta$ is the virtual vector bundle classified by the composition
	\begin{equation*}
		\mathbb{D}_{C_p}(\BGm)\xrightarrow{\mathbb{D}_{C_p}(\gamma_1)}\mathbb{D}_{C_p}(K^{\mathrm{hom}})\to K^{\mathrm{hom}}	
	\end{equation*} 
	where the last morphism is induced by the coherent multiplication on $K^{\mathrm{hom}}$.
\end{proposition}
\begin{proof}
	As $\Spc{-}_{/K^{\mathrm{hom}}}$ is a normed category, we have access to the norms functor, hence we can define the extended power functor as done in \cite{MotivicColimits}. Applying the motivic extended power $\mathbb{D}_{C_p}$ to the map $\gamma_1:\BGm\to K^{\mathrm{hom}}$ classifying the canonical bundle over $\BGm$, and using the normed multiplicative structure on $K^{\mathrm{hom}}$, we obtain a new morphism in $\mathbb{D}_{C_p}(\BGm)\to\mathbb{D}_{C_p}(K^{\mathrm{hom}})\in\Spc{S}_{/K^{\mathrm{hom}}}$, that is, a commutative diagram in $\Spc{S}$
	\[\begin{tikzcd}
		{\mathbb{D}_{C_p}(\BGm)} & & {\mathbb{D}_{C_p}(K^{\mathrm{hom}})} \\
		& & K^{\mathrm{hom}}
		\arrow["{\mathbb{D}_{C_p}(\gamma_1)}", from=1-1, to=1-3]
		\arrow["\beta"', from=1-1, to=2-3]
		\arrow[from=1-3, to=2-3]
	\end{tikzcd}\]
	Applying the motivic Thom spectrum functor, we obtain a commuting diagram in $\SH{S}$
	\[\begin{tikzcd}
		{M(\mathbb{D}_{C_p}(\BGm))} & & {M(\mathbb{D}_{C_p}(K^{\mathrm{hom}}))} \\
		& & {M(K^{\mathrm{hom}})}
		\arrow["{M(\mathbb{D}_{C_p}(\gamma_1))}", from=1-1, to=1-3]
		\arrow["{\mathrm{th}(\beta)}"', from=1-1, to=2-3]
		\arrow[from=1-3, to=2-3]
	\end{tikzcd}\]
	where we write $\mathrm{th}(\beta)$ for the Thom class of $\beta$.
	
	On the other hand we could also first apply the motivic Thom spectrum functor and then the extended powers, in which case we would have ended up with the diagram
	\[\begin{tikzcd}
		{\mathbb{D}_{C_p}(M(\BGm))} & & {\mathbb{D}_{C_p}(\PMGL)} \\
		& & \PMGL
		\arrow["{\mathbb{D}_{C_p}M(\gamma_1)}", from=1-1, to=1-3]
		\arrow["{\mathbb{D}_{C_p}(\mathrm{th}(\gamma_1))}"', from=1-1, to=2-3]
		\arrow[from=1-3, to=2-3]
	\end{tikzcd}\]
	in $\SH{S}$.
	
	Now we can use that by \cite[Remark 5.4]{MotivicColimits} the extended powers $\mathbb{D}_{C_p}$ and the Thom spectrum functor $M$ commute, identifying the two diagrams. We obtain an equivalence of morphisms
	\begin{equation*}
		\mathrm{th}(\beta)\simeq \mathbb{D}_{C_p}(\mathrm{th}(\gamma_1)).
	\end{equation*}
	Precomposition with the zero-section $s_0:\mathbb{D}_{C_p}(\Sigma^\infty_+\BGm)\to \mathbb{D}_{C_p}M(\BGm)$ and diagonal $\Delta_{\mathds{h}C_p}:(\Sigma^\infty_+\BGm)_{\mathds{h}C_n}\to \mathbb{D}_{C_p}(\BGm)$, we obtain an equality of cohomology classes
	\begin{equation*}
		\Delta_{\mathds{h}C_p}^\ast c_n(\beta)=c_n(\Delta_{\mathds{h}C_p}^\ast \beta)=
		P_{C_p}(c_1(\gamma_1))\in\PMGL^{0,0}((\BGm)_{\mathds{h}C_p}).
	\end{equation*}
	Here we used that the pullback of the Thom class along the zero section of a bundle recovers the top Chern class of that bundle, see \cite[Remark 2.4.5]{degliseorientation}.
\end{proof}
\begin{remark}
	We can visualize the equality of the cohomology classes by the following commutative diagram:
	\[\begin{tikzcd}[column sep = tiny, row sep = small]
		& {M(\BGm\times\Bet{C_p})} & {\mathbb{D}_{C_p}(\BGm)} \\
		{\sus(\BGm\times\Bet{C_p})} &&& {\mathbb{D}_{C_p}\PMGL} & \PMGL \\
		& {\mathbb{D}_{C_p}(\sus{\BGm\times\Bet{C_p}})} & {M\mathbb{D}_{C_p}(\BGm)}
		\arrow["\Delta", from=1-2, to=1-3]
		\arrow["{\mathbb{D}_{C_p}(\mathrm{th}(\gamma_1))}", from=1-3, to=2-4]
		\arrow["{s_0}", from=2-1, to=1-2]
		\arrow["\Delta"', from=2-1, to=3-2]
		\arrow["{\mu_\otimes}", from=2-4, to=2-5]
		\arrow["{s_0}"', from=3-2, to=3-3]
		\arrow["{\mathrm{th}(\beta)}"', from=3-3, to=2-4]
	\end{tikzcd}\]
	The upper composition exactly represents $P_{C_p}(c_1)$, while the lower composition represents $c_n(\Delta_{\mathds{h}C_p}^\ast \beta)$.
\end{remark}

\begin{proposition}
	There exists a $C_p$-equivariant isomorphism
	\begin{equation*}
		\gamma_1^{\oplus p}\ \cong\  \gamma_1\otimes \pi^\ast \rho,
	\end{equation*}
	where $\pi:\BGm\to \ast$ is the projection to the point, and $\rho$ is induced by the permutation representation of $C_p$, considers as a vector bundle over the point, that is, the stack represented by the base scheme $S$.
\end{proposition}
\begin{proof}
	We use that for a group $G$, the $G$-equivariant line bundles on the stack $\BGm$ are given by
	\begin{equation*}
		\mathrm{Vect}_1(\BGm)^G\simeq \mathrm{Fun}(\mathrm{B}G,\mathrm{Vect}_1(\BGm))\simeq \mathrm{lim}_{X\to\BGm}\mathrm{Vect}_1(X)^{G}.
	\end{equation*}
	Thus, to construct the desired isomorphism, we will construct an isomorphism of $C_p$-equivariant bundles on any schemes $X$, that is, on $X$-valued points, such that they glue.
	
	The points on the right hand side are given by $\mathcal{L}^{\oplus p}$, that is, by $p$-fold direct sums of a line bundle $\mathcal{L}$ on $X$, with the $C_p$-action given by permuting the factors. A point on the right hand side is $\mathcal{L}\otimes \mathcal{O}_X^{\oplus p}$, where $\mathcal{L}$ carries the trivial $C_p$-action, while $C_p$ acts on $\mathcal{O}_X^{\oplus p}$ by permuting the factors. But now we have the canonical, $C_p$ equivariant isomorphism
	\begin{equation*}
		\mathcal{L}^{\oplus p}\cong  \mathcal{L}\otimes \mathcal{O}_X^{\oplus p}
	\end{equation*}
	obtained from the fact that the tensor product commutes with direct sums. Since this isomorphism is an canonical, it is compatible with the base change of vector bundles. Hence this isomorphism glues to the desired isomorphism of vector bundles on the stack $\BGm$.
\end{proof}

The bundle $\rho$ is now accessible by character theory. In fact we have, $C_p$ being cyclic,
\begin{equation*}
	\rho\ \cong\ \bigoplus_{i=0}^{p-1}\chi^{\otimes i}
\end{equation*}
where $\chi$ is the generating character of $C_p$. We want to descend this isomorphism along the homotopy orbits. That this is possible is content of the next lemma.
\begin{lemma}
	Let $G$ be a finite group and let $V\to\mathcal{X}$ be a $G$-equivariant vector bundle over an algebraic stack, where $\mathcal{X}$ has the trivial action. Then $V_{\mathds{h}G}$ defines a vector bundle over $\mathcal{X}_{\mathds{h}G}=\mathcal{X}\times \Bet{G}$. Furthermore, this construction is compatible with direct sums and tensor products.
\end{lemma}
\begin{proof}
	It suffices to show that the motivic homotopy orbits agree on étale sheaves with the étale homotopy orbits. This is essentially the proof of \cite[Proposition 5.23]{MotivicColimits}, with the extend powers replaced by just the homotopy orbits.
\end{proof}
Thus, descending along homotopy orbits, we obtain
\begin{align*}
	\Delta^\ast_{\mathds{h}C_p}\beta\cong \bigoplus_{i=0}^{p-1}(\pi_1^\ast\gamma_1)\otimes\eta^{\otimes i}
\end{align*}
where $\eta=\pi_2^\ast(\chi_{\mathds{h}C_p})$ is the tautological line bundle over $\Bet{C_p}$, and $\pi_1$ and $\pi_2$ are the canonical projections of $\BGm\times\Bet{C_p}$. In summary, we obtain the equation,
\begin{equation*}
	P_{C_p}(c_1(\gamma_1))=c_p\left(\bigoplus_{i=0}^{p-1}(\pi_1^\ast\gamma_1)\otimes\eta^{\otimes i}\right)\in\PMGL^{0,0}((\BGm\times \Bet{C_p}),
\end{equation*}
using that by \cite[Lemma C.6]{MotivicColimits}, if $X$ has the trivial action, then $\sus X_{\mathds{h}C_p}\simeq(\sus X\otimes\Sph)_{\mathds{h}C_p}\simeq \sus(X\times\Bet{C_p})$.

By \Cref{rmk: MGLandPMGLpoweroperations}, the last equation is already true in $\MGL$-cohomology. We obtain, using multiplicativaty of the top Chern class and the formal group law on $\MGL$:
\begin{align*}
	P_{C_p}(c_1(\gamma_1))&=c_p(\Delta^{\ast}\beta)\\
	&=\prod_{i=0}^{p-1}(c_1(\gamma_1)+_{\MGL}[i]_{\MGL}(c_1(\eta))).
\end{align*}
Here, $+_{\MGL}$ is the addition for the formal group law on $\MGL$, that is, if the formal group law is given by the power series $f\in\MGL^{\ast,\ast}[[x,y]]$, then $x+_{\MGL} y$, while $[i]_{\MGL}$ is the multiplication with $i$ in the formal group law, that is $[i]_{\MGL}(x)=x+_{\MGL}x+_{\MGL}\ldots+_{\MGL}x$, where $x$ appears $i$-times on the right hand side.

By expanding out the formal group law, we obtain the following proposition.
\begin{proposition}
	\label{prop: formulaforfirstchernclass}
	There exists a universal formula for the power operation of the first Chern class
	\begin{equation*}
		P_{C_p}(c_1(\gamma_1))=c_1(\gamma_1)\sum_{|\alpha|\leq 1}w^{1-|\alpha|}a_\alpha(t) c_\alpha(\gamma_1),
	\end{equation*}
	where 
	\begin{equation*}
		w=\prod_{i=1}^{p-1}[i]_{\MGL}(t)
	\end{equation*}
	is the top Chern class of the reduced permutation representation. Furthermore, $t=c_1(\eta)$ is the top Chern class of the tautological bundle $\eta$ and the $a_\alpha(t)=\prod_{j=1}^\infty a_j(t)^{\alpha_j}$ are power series with coefficients in the subring of $\pi_\ast(\MGL)_\ast$ generated the coefficients of the formal group law on $\MGL$.
\end{proposition}
\begin{proof}
	The formula follows from the fact that multiplication distributes over the addition in the formal group law. By definition $c_\alpha(\gamma_1)$ is, for $|\alpha|=1$, a fixed power $c_1^{i}(\gamma_1)$ where $i$ is the place of the unique non-zero entry in $\alpha$. Further, $a_\alpha(t)$ is, again for $|\alpha|=1$, the power series 
	\begin{equation*}
		\underset{j_n\neq j_m \forall m,n}{\sum_{j_1,\ldots, j_{p-1}}}[j_1]_{\PMGL}(t)[j_2]_{\PMGL}(t)\cdots [j_{p-i}]_{\PMGL}(t)
	\end{equation*}
	where $i$ is as above.
\end{proof}
Before finishing the calculation of the universal class $u_1$ (we still have to remove the isomorphism coming from the zero section), we use \Cref{prop: formulaforfirstchernclass} to calculate all power operations $P_{C_p}(c_n(\gamma_n))\in\MGL(\Bet{GL(m)})$ and then remove all the zero section isomorphisms in one fell swoop to obtain the corresponding formulas for $P_{C_p}(u_n)$, where $u_n$ is the Thom class of the universal rank $n$ bundle on the classifying stack $\Bet{GL(n)}$.

\begin{remark}
	\label{exmpl: twisted oriented cohomology}
	For the next proposition, we remark that the following. If $F$ is any spectrum, $E$ an oriented homotopy ring spectrum, so is the twisted cohomology $E^F:=\mathrm{map}(F,E)$, with the orientation given by the composition $\MGL\to E\simeq \mathrm{map}(\Sph,E)\to\mathrm{map}(F,E)$. Here, the last map is induced by the unique morphism $F\to \Sph$. Notice that for any spectrum $Y$, we can describe the twisted cohomology as
	\begin{equation*}
		(E^F)^{\ast,\ast}(Y)=\pi_0\mathrm{map}(Y,\Sigma^{\ast,\ast} E^F)=\pi_0\mathrm{map}(Y,\Sigma^{\ast,\ast}\mathrm{map}(F,E))\cong\pi_0\mathrm{map}(Y\otimes F,\Sigma^{\ast,\ast}E)=E^{\ast,\ast}(F\otimes Y).
	\end{equation*}
	It follows from \cite[Proposition 6.2]{naumann2009motiviclandweberexactness}, that if $X$ is a scheme, or more generally a motivic space
	\begin{equation*}
		E^{\ast,\ast}(X\times \Bet{\mathrm{GL}(n)})\cong(E^X)^{\ast,\ast}(\mathrm{GL}(n))\cong (E^X)^{\ast,\ast}[[c_1,\ldots,c_n]]\cong E^{\ast,\ast}(X)[[c_1,\ldots,c_n]],
	\end{equation*}
	where we abbreviated $E^X:=E^{\sus X}$.
	
	For any oriented ring spectrum $E$ the image of the inclusion
	\begin{equation*}
		E^{\ast,\ast}[[c_1,\ldots,c_n]]\cong E^{\ast,\ast}(\Bet{\mathrm{GL}(n)})\hookrightarrow E^{\ast,\ast}(\BGm^{\times n})\cong E^{\ast,\ast}[[x_1,\ldots,x_n]],
	\end{equation*}
	can be identified with the subring of symmetric power series on the right hand side. By the usual Milnor sequence argument, it follows that the image of the inclusion 
	\begin{equation*}
		E^{\ast,\ast}[[c_1,\ldots]]\hookrightarrow E^{\ast,\ast}[[x_1,\ldots]],
	\end{equation*}
	is also given by the symmetric power series.
	
	From these facts, we conclude that the morphism $\iota^\ast:E^{\ast,\ast}(X\times \Bet{\mathrm{GL}(n)})\to E^{\ast,\ast}(X\times \Bet{\mathrm{GL}(1)}^{\times n})$ is injective for all motivic spaces $X$.
	
	Let $E$ be again a complex oriented ring spectrum. Consider the Thom isomorphism 
	\begin{equation*}
		\varphi:E^{\ast,\ast}(\Bet{\mathrm{GL}}(n))\xrightarrow{\cong} E^{\ast,\ast}(\MGL(n)),	
	\end{equation*}
	which is explicitly given by multiplication with the Thom class $\mathrm{th}(\gamma_n)$. Denote the zero-section of the canonical line bundle $\gamma_n$ over $\Bet{\mathrm{GL}(n)}$ by $s_0$. The pullback along $s_0$ is injective in $E^{\ast,\ast}$-cohomology. Indeed, using the Thom isomorphism, and that $s_0^\ast(\mathrm{th}(\gamma_n))=c_n(\gamma_n)$, we see that the composition
	\begin{equation*}
		E^{\ast,\ast}(\Bet{\mathrm{GL}}(n))\xrightarrow{\varphi}E^{\ast,\ast}(\MGL(n))\xrightarrow{s_0^\ast}E^{\ast+2n,\ast+n}(\Bet{\mathrm{GL}(n)})
	\end{equation*}
	is given by the multiplication by $c_n(\gamma_n)$, which is a non-zero divisor in the power series ring $E^{\ast,\ast}(\Bet{\mathrm{GL}(n)})\cong E^{\ast,\ast}[[c_1,\ldots,c_n]]$. As $\varphi$ is an isomorphism, $s_0^\ast$ is injective.
\end{remark}

\begin{proposition}
	\label{prop: formulafornththomclass}
	For any $n\in\mathbb{N}$ there exists a universal formula
	\begin{equation*}
		P_{C_p}(u_n)=\sum_{|\alpha|\leq n}w^{n-|\alpha|}a_\alpha(t) s_\alpha(u_n)\in\MGL^{2np,np}(\Bet{GL(n)}\times\Bet{C_p})
	\end{equation*}
	where $s_\alpha$ are the motivic Landweber-Novikov operations, see \Cref{defn: Landwebernovikovoperationsmotivically}, and the $a_\alpha$ are power series with coefficients in $C\subseteq \MGL^{\ast,\ast}$, the subring generated by the coefficients of the formal group law on $\MGL$.
\end{proposition}
\begin{proof}	
	Denote the Thom isomorphism by $\varphi:\MGL^{\ast,\ast}(\Bet{\mathrm{GL}}(n))\xrightarrow{\cong} \MGL^{\ast,\ast}(\MGL(n))$. Using \Cref{exmpl: twisted oriented cohomology} with $E=\MGL$ and $X=S=\ast\in\Spc{S}$, respectively $X=\Bet{C_p}$, and the naturality of $P_{C_p}$, we obtain the following commutative diagram.
	\[\begin{tikzcd}[column sep = small]
		{\MGL^{\ast,\ast}(\mathrm{M}\mathrm{GL}(n)\otimes\Sigma_+^\infty\Bet{C_p})} & {\MGL^{\ast,\ast}(\Bet{\mathrm{GL}}(n)\times\Bet{C_p})} & {\MGL^{\ast,\ast}(\Bet{\mathrm{GL}}(1)^{\times n}\times\Bet{C_p})} \\
		{\MGL^{\ast,\ast}(\mathrm{M}\mathrm{GL}(n))} & {\MGL^{\ast,\ast}(\Bet{\mathrm{GL}}(n))} & {\MGL^{\ast,\ast}(\Bet{\mathrm{GL}}(1)^{\times n}),}
		\arrow["{s_0^\ast}", hook, from=1-1, to=1-2]
		\arrow["{\iota^\ast}", hook, from=1-2, to=1-3]
		\arrow["{P_{C_p}}"', from=2-1, to=1-1]
		\arrow["{s_0^\ast}", hook, from=2-1, to=2-2]
		\arrow["{P_{C_p}}"', from=2-2, to=1-2]
		\arrow["{\iota^\ast}", hook, from=2-2, to=2-3]
		\arrow["{P_{C_p}}"', from=2-3, to=1-3]
	\end{tikzcd}\]
	where all horizontal morphisms are injective. It follows from injectivity, that it suffices to calculate $P_{C_p}$ on the image of $u_m$ in the lower right corner, that is, on
	\begin{equation*}
		\iota^\ast s_0^\ast u_m=\iota^\ast c_n(\gamma_n)=\prod_{i=1}^nc_1(\xi_i)\in\MGL^{\ast,\ast}(\mathrm{B}U(1)^{\times n}),
	\end{equation*}
	where we write $\xi_i:=\pi_i^\ast \gamma_1$. Here we used that the pullback bundle $\iota^\ast\gamma_n$ splits as a product of line bundles on $\Bet{\mathrm{GL}}(1)^{\times n}$, so that we can use splitting principle, i.e.\! the multiplicativity of the Chern classes in short exact sequences of bundles, $\iota^\ast c_n(\gamma_n)=\prod_{i=1}^nc_1(\xi_i)$, see \cite[Paragraph 1.2.10]{deglise2024characteristic} and \cite[page 51]{fulton1998intersection}.
	
	We now calculate
	\begin{align*}
		\iota^\ast s_0^\ast P_{C_p}(u_n)&=P_{C_p}\left(c_1(\xi_1)\cdots c_1(\xi_n)\right)\\
		&=P_{C_p}\left(c_1(\xi_1)\right)\cdots (P_{C_p}\left(c_1(\xi_n)\right)\\
		&=(c_1(\xi_1)\sum_{|\alpha|\leq 1}w^{1-|\alpha|}a^{\alpha}(t)c_\alpha(\xi_1))\ldots(c_1(\xi_n)\sum_{|\alpha|\leq 1}w^{1-|\alpha|}a^{\alpha}(t)c_\alpha(\xi_n))\\
		&=c_n(\xi_1\times\ldots\times \xi_n)\sum_{|\alpha|\leq m}w^{m-\alpha}a_{\alpha}(t)c_\alpha(\xi_1\times\ldots\times \xi_n)\\
		&=c_n(\iota^\ast\gamma_n)\sum_{|\alpha|\leq m}w^{m-\alpha}a_{\alpha}(t)c_\alpha(\gamma_n)\\
		&=\iota^\ast(c_n(\gamma_n)\sum_{|\alpha|\leq m}w^{m-\alpha}a_{\alpha}(t)c_\alpha(\gamma_n)).
	\end{align*}
	Here we used the multiplicativity of the power operation.
	
	After canceling $\iota^\ast$ due to injectivity, we see that, by the last paragraph of \Cref{exmpl: twisted oriented cohomology}, the factor $c_n(\gamma_n)$ implies that the sum 
	\begin{equation*}
		\sum_{|\alpha|\leq m}w^{m-\alpha}a_{\alpha}(t)c_\alpha(\gamma_n)
	\end{equation*}
	is in the image of $s_0\circ\varphi$, hence originates from $\MGL^{\ast,\ast}(\Bet{\mathrm{GL}}(n)\times\Bet{C_p})$, that is
	\begin{equation*}
		\sum_{|\alpha|\leq m}w^{m-\alpha}a_{\alpha}(t)c_\alpha(\gamma_n)\in \MGL^{\ast,\ast}(\Bet{\mathrm{GL}}(n)\times\Bet{C_p}).
	\end{equation*}
	Commutativity of the diagram and injectivity of all involved maps implies that this sum is mapped, under the Thom isomorphism, to $P_{C_p}(u_m)$. We obtain the desired equality
	\begin{align*}
		P_{C_p}(u_n)&=\varphi\left(\sum_{|\alpha|\leq n}w^{n-\alpha}a_{\alpha}(t)c_\alpha(\gamma_n)\right)\\
		&=\sum_{|\alpha|\leq n}w^{n-\alpha}a_{\alpha}(t)s_\alpha(u_n).
	\end{align*}
\end{proof}

\subsection{Stabilizing the Power Operations}
The next step in the topological proof is to study how the power operation behaves with respect to suspension, so let us also study suspensions next. 
\begin{proposition}
	\label{prop: stabilityofpoweroperations}
	Let $X$ be a compact object in $\Spc{S}_\ast$ and $x\in \MGL^{2n,n}(X)$. Then there is a suspension relation for the power operation
	\begin{equation*}
		P_{C_p}(\sigma^{n}_{\mathbb{T}}x)=\sigma^n_{\mathbb{T}}w^{n}P_{C_p}(x)\in \MGL^{2np,np}(X\times\Bet{C_p})
	\end{equation*}
	where $\sigma_{\mathbb{T}}$ is the suspension homomorphism coming from the Tate motive $\mathbb{T}=S^{2,1}$.
\end{proposition}
\begin{proof}
	Consider the unit $1\in\MGL^{0,0}(S^{0,0})$. Suspending this class $n$-times we obtain $\iota_{2n,n}=\iota_{2n,n}^{\ast}u_n\in\MGL^{2n,n}(S^{2n,n})$, represented by a morphism
	\begin{equation*}
		\mathbb{T}^{n}\xrightarrow{\iota_{2n,n}}M(\Bet{GL}(n))\xrightarrow{u_n}\Sigma^{2n,n}\MGL.
	\end{equation*}
	Applying the power operation to this class, and using that $P_{C_p}$ is compatible with pullbacks, we obtain
	\begin{equation*}
		P_{C_p}(\iota_{2n,n}^{\ast}x_n)=\iota_{2n,n}^{\ast}P_{C_p}(x_n)\overset{\ref{prop: formulafornththomclass}}{=}\iota_{2n,n}^{\ast}\sum_{|\alpha|\leq n}w^{n-|\alpha|}a_\alpha s_\alpha(x_n)=w^n\iota_{2mn,mn}.
	\end{equation*}
	Here we use that $\iota^\ast_{2n,n}s_\alpha(x_n)=s_\alpha(\iota^\ast_{2n,n}(x_n))\in \MGL^{2n+2|\alpha|,n+|\alpha|}(S^{2n,n})$, which is zero for $|\alpha|>0$ by connectivity of $\MGL$, see \cite[Corollary 3.9]{hoy15}.
	
	Now take an $x\in\MGL^{2n,n}(X)$ as in the statement of the proposition. Then we have $\sigma_{\mathbb{T}}^nx=\iota_{2n,n}\wedge x$ so the claim follows from the multiplicativity of the power operations, \Cref{rmk: multiplicativity PG}.
\end{proof}
This is again an almost stability statement for $P_{C_p}$, equivalent to the topological case.

\begin{proposition}
	\label{Prop: GeneralformulaforPCp}
	Let $X\in\Spc{S}_\ast$ be a compact, pointed motivic space and $x\in\widetilde{\MGL}^{2n,n}(X)$. Then there is a formula for $m\gg 0$
	\begin{equation*}
		w^mP_{C_p}(x)= \sum_{|\alpha|\leq n+m}w^{n+m-|\alpha|}a_\alpha(t) s_\alpha(x)\in\MGL^{2(np+m-1),np+m-1}(X\times\Bet{C_p}).
	\end{equation*}
\end{proposition}
\begin{proof}
	Let $x:\sus X\to \Sigma^{2n,n}\MGL$. As $X$ is compact by assumption, there exists an $m\gg0$ and an unstable morphism
	\begin{equation*}
		f:\Sigma^{2m,m} X\to M(\Bet{GL(n+m)}).
	\end{equation*}
	This morphism satisfies that
	\begin{equation*}
		f^\ast x_{n+m}=\sigma_{\mathbb{T}}^{m}x\in\MGL^{\ast,\ast}(X)
	\end{equation*}
	where $x_{n+m}$ is still the Thom class of the tautological bundle on $\Bet{GL(n+m)}$. Now we simply calculate using Proposition \ref{prop: stabilityofpoweroperations} and Proposition \ref{prop: formulafornththomclass}
	\begin{align*}
		w^{m}\sigma_{\mathbb{T}}^mP_{C_p}(x)&=P_{C_p}(\sigma_{\mathbb{T}}^mx)\\
		&=P_{C_p}(f^\ast x_{n+m})\\
		&=f^\ast P_{C_p}( x_{n+m})\\
		&=f^{\ast} \left(\sum_{|\alpha|\leq n+m}w^{n+m-|\alpha|}a_\alpha(t) s_\alpha(x_{n+m})\right)\\
		&=\sigma_{\mathbb{T}}^{m}\sum_{|\alpha|\leq n+m}w^{n+m-|\alpha|}a_\alpha(t) s_\alpha(x)
	\end{align*}
	which was the claim, using that $\sigma_{\mathbb{T}}^{m}:\MGL^{\ast,\ast}(X)\to \MGL^{\ast+2m,\ast+m}(\Sigma^{2m,m}X)$ is an isomorphism.
\end{proof}
\begin{remark}
	\label{rmk: Havettorsion}
	At this point we should point out, that even though this looks like a nice formula, it does not give us a complete solution. Notice that the power series $a_\alpha$ all depend on the top Chern class $t$ of the tautological bundle $\eta$, which lives over $\Bet{C_p}$. So we need to understand the $t$-torsion in $\Bet{C_p}$.
\end{remark}
	\subsection{\texorpdfstring{Oriented Cohomology of the stack $\Bet{\mu_p}$}{Oriented Cohomology of the Stack Betmup}}
We now want to prove a motivic version of Quillen's ``technical lemma'', see \cite[Section 4]{BordismQuillen}. 

\begin{definition}
	Let $x$ be a coordinate of a formal group law $F$. Then we define the \emph{divided n-series} $\langle n\rangle(x)$ by the equality
	\begin{equation*}
		x\ \cdot\langle n\rangle_F(x)=[n]_F(x).
	\end{equation*}
	If the formal group law under consideration is clear from the context, we suppress the $F$.
\end{definition}

To get control over the $t$-dependence mentioned in \Cref{rmk: Havettorsion}, we need to replace $C_p$ with the group scheme of $p$-th roots of unity $\mu_p$.

To be able to do so, we will from now on assume that our base scheme $S=\mathrm{Spec}(k)$ is the spectrum of a field $k$ containing a primitive $p$-th root of unity $\zeta$. We can hence identify $C_p$ with $\mu_p$, using an explicit choice of a primitive $p$-th root of unity. In particualr, the action of $C_p$ on $\Gm$, and hence on $\A^1$, is given by the character $C_p\cong \mu_p\hookrightarrow\Gm$ and is hence given by multiplication with the chosen primitive $p$-th root of unity.

The restriction to fields containing a primitive $p$-th root of unity does not pose a great loss for us, since we are later going to assume the base field to even have all primitive $p$-th power roots of unity. Notice, that $\zeta\in k$ implies that the characteristic of $k$ is not $p$.

Throughout the rest of the section, denote $U_n:=(\A^n\setminus\{0\})$.
\begin{lemma}
	\label{lem: Cptorsordividedpseries}
	Let $n\geq 1$ and consider the $\mu_p$-torsor $f:U_n\to U_n/\mu_p\in\Sm_{k}$, where we consider $U_n$ with the diagonal $\mu_p$-action coming from the character $\mu_p\hookrightarrow \Gm$ and the usual action of $\Gm$ on $\A^1$ by left multiplication. Let $L:=U_n\times_{\mu_p}\A^1:=(U_n\times\A^1)/\mu_p\to U_n/\mu_p$ be the associated line bundle, where the action on $\A^1$ is again determined by the character $\mu_p\hookrightarrow\Gm$, and the action on the product $U_n\times \A^1$ is given by the diagonal action. Let $E$ be an incoherently normed oriented ring spectrum. Then we obtain
	\begin{equation*}
		f_\ast(1)=\langle p\rangle_E(t_n) \in E^{2,1}(U_n/\mu_p),
	\end{equation*}
	where $t_n$ is the Euler class of the canonical quotient bundle $U_n\to U_n/\mu_p$.
\end{lemma}
\begin{proof}
	We follow closely Quillen's original argument \cite[Proposition 4.2]{BordismQuillen}. We implicitly use the functor-of-points point of view, in particular to define mpas. Notice that in this case, the maps have to be defined section wise in such a way, that they can be glued.
	
	Before we delve into the proof, we first note that the quotients $U_n/\mu_p$, where $\mu_p$ acts via the canonical action given by left multiplication, and $U_n\times_{\mu_p}\Gm:=(U_n\times\Gm)/\mu_p$, as well as $U_n\times_{\mu_p}\A^1$, where $\mu_p$ acts in both cases via the diagonal action, are all schemes for $n>0$. This is true by \cite[III.3.12, Proposition 18]{serre1988algebraic}, since $\mu_p$ is finite, we are working over a field, and all schemes we take the quotient of are varieties.
	
	As the action of $\mu_p$ on $U_n$, given by the left multiplication on each factor, is free there is an equivalence $U_n/\mu_p\simeq U_n\sslash\mu_p$, where $(-)\sslash\mu_p$ denotes the homotopy quotient in étale sheaves. Thus, the $\mu_p$ torsor $U_n\to U_n/\mu_p$ corresponds to a $\mu_p$ equivariant map $U_n\times \mu_p\to U_n$, and similarly $L\to U_n/\mu_p$ corresponds to a $\mu_p$ equivariant map $\A^1_{U_n}=U_n\times\A^1\to U_n$. 
	
	We have a canonical, $\mu_p$ equivariant morphism\footnote{This corresponds to a map $j:U_n\to L$, where points in the scheme $L$ are then given by $j(q)z$, with $q\in U_n$, $z\in \A^1$, satisfying the relation $j(q)(\zeta z)=j(\sigma q)z$, where $\sigma$ is a generator of $\mu_p$ and $\zeta$ is a primitive $p$-th root of unity. This is exactly the morphism used in Quillen's original proof.} $j:U_n\times\mu_p\hookrightarrow \A^1_{U_n}$ given by inclusion of the second factor. We write points in $\A^1_{U_n}$ as $(q,z)$, with $q\in U_n$ and $z\in\A^1$. We further denote by $s_0:U_n\to\A^1_{U_n}$ the zero section, and by $p:\A^1_{G_m}\to\Gm$ the projection.
	
	Consider now the pullback bundle 
	\begin{equation*}
		p^\ast \A^1_{U_n}= \A^1_{U_n}\times_{U_n}\A^1_{U_n}=\A^1_{\A^1_{U_n}}\to \A^1_{U_n}.	
	\end{equation*}
	This admits a tautological section $s:\A^1_{U_n}\to p^\ast \A^1_{U_n}$, given by the relative diagonal $s=\Delta_{p}$. We point out, that the vanishing locus of $s$ is given by $s_0(U_n)$, as can be seen directly from the description $s(q,z)=((q,z),(q,z))\in p^\ast \A^1_{U_n}$.
	
	Consider the $p$-fold tensor product $(\A^1_{\A^1_{U_n}})^{\otimes p}$ which is again a $\mu_p$-equivariant line bundle over $\A^1_{U_n}$. We explicitly define a $\mu_p$-equivariant section as
	\begin{equation*}
		t:\A^1_{U_n}\to (p^\ast \A^1_{U_n})^{\otimes p}, \quad (q,z)\mapsto ((q,z),((q,z)^{\otimes p}-(q,1)^{\otimes p}).
	\end{equation*}
	This defines a well-defined morphism of schemes, as the section is polynomial and hence can be glued. From the explicit description, we can deduce the vanishing locus of $t$ to be the image of $j$, since the vanishing locus is exactly given by solutions to the equation $z^p-1$ in the field $k$.
	
	We now extend $\A^1_{U_n}$ to $\mathbb{P}^1_{U_n}$ by adding the point at infinity. We claim that we can extend $p^\ast \A^1_{U_n}$ to a $\mu_p$-equivariant vector bundle\footnote{As the vector bundle was extended by the usual gluing diagram for $\mathbb{P}^1$ we see that the locally free sheaf corresponding to the line bundle $M$ is given by $\mathcal{O}(1)$.} $M$ over $\mathbb{P}_{U_n}^1$, with transition morphism induced by the section $s$. But this is clear: We take $L_1:=\A^1_{U_n}\times(\mathbb{P}^1_{U_n}\setminus \infty)$ and  $L_2:=\A^1_{U_n}\times(\mathbb{P}^1_{U_n}\setminus 0)$, and define the isomorphism $\varphi:L_1|_{\mathbb{G}_m}\to L_2|_{\mathbb{G}_m}$ by $\varphi(z,x):=(z,\tfrac{x}{z})$.
	
	We can extend the bundle $(p^\ast \A^1_{U_n})^{\otimes p}$ over $\mathbb{P}^1_{U_n}$ to the $\mu_p$-equivariant bundle $M^{\otimes p}$ in a similar fashion. Further, the section $t$ extends to a section $t:\mathbb{P}^1_{U_n}\to M^{\otimes p}$. In summary, we obtain the following $\mu_p$-equivariant diagram:
	\[\begin{tikzcd}
		& {(p^\ast\A^1_{U_n})^{\otimes p}} & {M^{\otimes p}} \\
		{U_n\times\mu_p} & {\A^1_{U_n}} & {\mathbb{P}^1_{U_n}} \\
		& {U_n}
		\arrow[hook, from=1-2, to=1-3]
		\arrow[from=1-2, to=2-2]
		\arrow[from=1-3, to=2-3]
		\arrow["j", hook, from=2-1, to=2-2]
		\arrow["f"', from=2-1, to=3-2]
		\arrow["t", shift left, curve={height=-6pt}, from=2-2, to=1-2]
		\arrow[hook, from=2-2, to=2-3]
		\arrow["p", from=2-2, to=3-2]
		\arrow["t", shift left, curve={height=-6pt}, from=2-3, to=1-3]
		\arrow["s_0"', shift right, curve={height=6pt}, from=3-2, to=2-2]
	\end{tikzcd}\]
	
	Taking $\mu_p$ homotopy orbits in étale sheaves, we obtain $(U_n\times\mu_p) \sslash\mu_p\simeq U_n$, $U_n \sslash\mu_p\simeq U_n/\mu_p$, and $\A^1_{U_n}\sslash\mu_p\simeq \eta_n$. For brevity, we write furthermore $\mathbb{P}$ and $M'$ for the $\mu_p$ quotients of $\mathbb{P}^1_{U_n}$, and $M$ respectively.
	
	Using \cite[Corollary 5.21, see also Meta-Theorem 3.9]{bachmann2023euler}, we can identify the Thom isomorphism with the pushforward of the zero section $s_0:U_n/\mu_p \to \eta_n$:
	\begin{equation*}
		(s_0)_\ast:E^{\ast,\ast}(U_n/\mu_p)\xrightarrow{\cong} \tilde{E}^{\ast+2,\ast+1}(\mathrm{Th}(\eta_n))=E^{\ast+2,\ast+1}(\mathbb{P},\infty). 
	\end{equation*}
	Explicitly, $(s_0)_\ast(1)=\mathrm{th}(\eta_n)=\tilde{t}_n$ is the Thom class of the line bundle $\eta_n\to U_n/\mu_p$. By construction $s_0(\mu_p/\mu_p)$ is the zero locus\footnote{Here we use that taking homotopy orbits of group schemes commutes with taking pullbacks, using that by \cite[Theorem 4.46]{Vistoli2005}, applying homotopy orbits defines an equivalence of categories between equivariant schemes and schemes over the stacky quotient,} of the section $s:\mathbb{P}\to M'$. So, using \cite[Corollary 5.21]{bachmann2023euler}, we also have that $(s_0)_\ast(1)=e(M')\in E^{\ast,\ast}(\mathrm{Th}(\eta_n))$. Similarly, $j_\ast(1)=e((M')^{\otimes p})\in E^{\ast,\ast}(\mathrm{Th}(\eta_n))$, where $j:U_n \to \mathbb{P}$. We can thus calculate
	\begin{align*}
		(s_0)_\ast(f_n)_\ast (1) &=  j_\ast(1)\\
		&=e((M')^{\otimes p})\\
		&=[p]_E(e(M'))\\
		&=[p]_E((s_0)_\ast(1))\\
		&=(s_0)_\ast(1)\langle p\rangle_E((s_0)_\ast(1)))\\
		&=(s_0)_\ast\langle p\rangle_E(s_0^\ast(s_0)_\ast(1))).
	\end{align*}
	Here we used in the third line that $e((M')^{\otimes p})=c_1((M')^{\otimes p})=[p]_E(c_1(M'))=[p]_E(e(M'))$ by the orientation on $E$ and the fact, that $M'$ is a line bundle. In the fifth line, we used the definition of the divided $p$-series $[p]_E(x)=x\langle p\rangle_E(x)$. The last line is the projection formula for closed immersions.
	
	Now $(s_0)_\ast$ is an isomorphism, and as $s_0^\ast(s_0)_\ast(1)=s_0^\ast\tilde{t}(\eta_n)=e(\eta_n)=t_n$, we obtain the desired result.
\end{proof}

\begin{remark}
	Notice that the proof of this proposition in particular shows that 
\end{remark}

Our next goal is to prove Quillen's technical lemma.
\begin{proposition}[Quillen's technical lemma]
	\label{prop: scheme version of quillens technical lemma}
	Let $E$ be an incoherently normed, oriented spectrum. Denote by $j_n:U_n/\mu_p\to U_n/\mu_p$ the induced map from the inclusion at $1$.
	
	If there exists an element $x\in E^{\ast,\ast}(U_{n+1}/\mu_p)$ with the property that $xt_{n+1}=0$, then there exists an element $y\in E^{\ast,\ast}$ such that 
	\begin{equation*}
		y\langle p\rangle_E(t_n)=j_n^\ast x\in E^{\ast,\ast}(U_n/\mu_p).
	\end{equation*}
\end{proposition}
We start with some prerequisites. We will use the identification $\A^{1}\setminus\{0\}\cong\Gm$ without further mentioning.
\begin{lemma}
	\label{lem: gysin sequences}
	There exists a commutative diagram of Gysin sequences
	\begin{equation*}
		\begin{tikzcd}
			\cdots & {E^{\ast,\ast}(\Gm/\mu_p)} & {E^{\ast,\ast}(U_{n+1}\times_{\mu_p}\Gm)} & {E^{\ast-2(n+1)+1,\ast-(n+1)+1}(\Gm/\mu_p)} & \cdots \\
			\cdots & {E^{\ast,\ast}(\Gm/\mu_p)} & {E^{\ast,\ast}(U_n\times_{\mu_p}\Gm)} & {E^{\ast-2n+1,\ast-n+1}(\Gm/\mu_p)} & \cdots.
			\arrow["{t_1^{n+1}}", from=1-1, to=1-2]
			\arrow["{(p_{n+1})^\ast}", from=1-2, to=1-3]
			\arrow["{\mathrm{id}}"', from=1-2, to=2-2]
			\arrow["{(p_{n+1})_\ast}", from=1-3, to=1-4]
			\arrow["{(j_n')^\ast}"', from=1-3, to=2-3]
			\arrow["{t_1^{n+1}}", from=1-4, to=1-5]
			\arrow["t_1",from=1-4, to=2-4]
			\arrow["{t_1^n}"', from=2-1, to=2-2]
			\arrow["{(p_{n})^\ast}"', from=2-2, to=2-3]
			\arrow["{(p_n)_\ast}"', from=2-3, to=2-4]
			\arrow["{t_1^{n}}"', from=2-4, to=2-5]
		\end{tikzcd}
	\end{equation*}
	Here $p_n:U_n\times_{\mu_p}\Gm$ is induced from the projection onto the last factor, and $j'_n:U_n\times_{\mu_p}\Gm\to U_{n+1}\times_{\mu_p}\Gm$ is induced from the inclusion at $0$.
	
	Furthermore, $t_1$, the Euler class of the canonical quotient bundle $\A^1\times_{\mu_p}\Gm\to\Gm$ vanishes, so that the above diagram splits into a morphism of short exact sequences
	\begin{equation*}
		\begin{tikzcd}
			0 & {E^{\ast,\ast}(\Gm/\mu_p)} & {E^{\ast,\ast}(U_{n+1}\times_{\mu_p}\Gm)} & {E^{\ast-2(n+1)+1,\ast-(n+1)+1}(\Gm/\mu_p)} & 0 \\
			0 & {E^{\ast,\ast}(\Gm/\mu_p)} & {E^{\ast,\ast}(U_{n+1}\times_{\mu_p}\Gm)} & {E^{\ast-2n+1,\ast-n+1}(\Gm/\mu_p)} & 0.
			\arrow[from=1-1, to=1-2]
			\arrow["{(p_{n+1})^\ast}", from=1-2, to=1-3]
			\arrow["{\mathrm{id}}"', from=1-2, to=2-2]
			\arrow["{(p_{n+1})_\ast}", from=1-3, to=1-4]
			\arrow["{(j'_n)^\ast}"', from=1-3, to=2-3]
			\arrow[from=1-4, to=1-5]
			\arrow["0", from=1-4, to=2-4]
			\arrow[from=2-1, to=2-2]
			\arrow["{(p_{n})^\ast}"', from=2-2, to=2-3]
			\arrow["{(p_n)_\ast}"', from=2-3, to=2-4]
			\arrow[, from=2-4, to=2-5]
		\end{tikzcd}
	\end{equation*}
\end{lemma}
\begin{proof}
	Notice first, that $t_1$ is the image of an element in $\MGL^{2,1}(\Gm/\mu_p)$. However, by the Kummer sequence, we know $\Gm/\mu_p\simeq\Gm$ with the equivalence being induced by the $p$-th power map, so we have
	\begin{align*}
		\MGL^{2,1}(\Gm/\mu_p)&=\pi_0\mathrm{map}(\sus\Gm,\Sigma^{2,1}\MGL)\\
		&\cong \pi_0\mathrm{map}(\Sigma^{1,1}\Sph,\Sigma^{2,1}\MGL)\\
		&\cong \pi_0\mathrm{map}(\Sph,\Sigma^{1,0}\MGL)\\
		&\cong \pi_{-1}(\MGL)_0\\
		&=0
	\end{align*}
	as $\MGL$ is connective. This takes care of the ``Furthermore'' part, once we have shown the existence of the diagram of Gysin sequences. For this, consider the commutative diagram of étale sheaves
	\[\begin{tikzcd}
		& {0\times \Gm} && {\A^{n}\times \Gm} && {U_n\times\Gm} \\
		& {\Gm/\mu_p} && {\A^n\times_{\mu_p}\Gm} && {U_n\times_{\mu_p}\Gm} \\
		{0\times \Gm} && {\A^{n+1}\times \Gm} && {U_{n+1}\times\Gm} \\
		{\Gm/\mu_p} && {\A^{n+1}\times_{\mu_p}\Gm} && {U_{n+1}\times_{\mu_p}\Gm}
		\arrow["\shortmid"{marking}, hook, from=1-2, to=1-4]
		\arrow[from=1-2, to=2-2]
		\arrow[from=1-2, to=3-1]
		\arrow[from=1-4, to=2-4]
		\arrow["\bullet"{marking, text=\pgfkeysvalueof{/tikz/commutative diagrams/background color}}, "\circ"{marking}, hook', from=1-6, to=1-4]
		\arrow[from=1-6, to=2-6]
		\arrow["\shortmid"{marking}, hook,  from=2-2, to=2-4]
		\arrow[from=2-2, to=4-1]
		\arrow[from=2-4, to=4-3]
		\arrow["\bullet"{marking, text=\pgfkeysvalueof{/tikz/commutative diagrams/background color}}, "\circ"{marking}, hook', from=2-6, to=2-4]
		\arrow[hook', from=2-6, to=4-5]
		\arrow[from=3-1, to=4-1]
		\arrow[from=3-3, to=4-3]
		\arrow[from=3-5, to=4-5]
		\arrow["\shortmid"{marking}, hook,  from=4-1, to=4-3]
		\arrow["\bullet"{marking, text=\pgfkeysvalueof{/tikz/commutative diagrams/background color}}, "\circ"{marking}, hook', from=4-5, to=4-3]
		\arrow[hook', from=1-4, to=3-3, crossing over]
		\arrow[hook', from=1-6, to=3-5, crossing over]
		\arrow["\shortmid"{marking}, hook, from=3-1, to=3-3,crossing over]
		\arrow["\bullet"{marking, text=\pgfkeysvalueof{/tikz/commutative diagrams/background color}}, "\circ"{marking}, hook', from=3-5, to=3-3,crossing over]
	\end{tikzcd}\]
	where the horizontal arrows are closed immersions with their open complements, while the vertical arrows are the canoncial projections onto the quotients. Applying the functoriality of the Gysin sequence in closed pairs, see \cite[Section 3]{degliseorientation} and \cite{deglise2021fundamental}, we obtain the following diagram from the bottom layer of the diagram
	\begin{equation*}
		\begin{tikzcd}
			\cdots & {E^{\ast,\ast}(\Gm/\mu_p)} & {E^{\ast,\ast}(U_{n+1}\times_{\mu_p}\Gm)} & {E^{\ast-2(n+1)+1,\ast-(n+1)+1}(\Gm/\mu_p)} & \cdots \\
			\cdots & {E^{\ast,\ast}(\Gm/\mu_p)} & {E^{\ast,\ast}(U_n\times_{\mu_p}\Gm)} & {E^{\ast-2n+1,\ast-n+1}(\Gm/\mu_p)} & \cdots
			\arrow["{t_1^{n+1}}", from=1-1, to=1-2]
			\arrow["{(p_{n+1})^\ast}", from=1-2, to=1-3]
			\arrow["{\mathrm{id}}"', from=1-2, to=2-2]
			\arrow["{(p_{n+1})_\ast}", from=1-3, to=1-4]
			\arrow["{(j_n')^\ast}"', from=1-3, to=2-3]
			\arrow["{t_1^{n+1}}", from=1-4, to=1-5]
			\arrow[from=1-4, to=2-4]
			\arrow["{t_1^n}"', from=2-1, to=2-2]
			\arrow["{(p_{n})^\ast}"', from=2-2, to=2-3]
			\arrow["{(p_n)_\ast}"', from=2-3, to=2-4]
			\arrow["{t_1^{n}}"', from=2-4, to=2-5]
		\end{tikzcd}
	\end{equation*}
	where we already used $\A^1$-invaraiance to rewrite $E^{\ast,\ast}(\A^n\times_{\mu_p}\Gm)\simeq E^{\ast,\ast}(\Gm/\mu_p)$. We want to identify the right most vertical arrow with the multiplication by $t_1$. Since we already know that $t_1=0$, the maps $(p_n)_\ast$ are surjective for all $n$ so it suffices to identify the right most vertical morphism in the Gysin sequence diagram after precomposition by $(p_{n+1})_\ast$. By commutativity, this is equivalent to studying $(p_n)_\ast (j'_n)^\ast$. Using \Cref{lem: helplemma commutativediagram}, we can calculate for any $z\in E^{\ast,\ast}(U_{n+1}\times_{\mu_p}\Gm)$
	\begin{align*}
		(p_{n})_\ast (j'_n)^\ast z&=(p_{n+1})_\ast (j'_n)_\ast (j'_n)^\ast z\\
		&=(p_{n+1})_\ast (j'_n)_\ast (1\cdot ((j'_n)^\ast z))\\
		&=(p_{n+1})_\ast ((j'_n)_\ast(1)\cdot z).
	\end{align*}
	It follows that the right most vertical morphism in the Gysin sequence diagram is given by multiplication with $(j'_n)_\ast(1)$. Using \Cref{lem: helplemma commutativediagram} again, together with the fact that $(p_{n})^\ast$ is a ring homomorphism, we obtain
	\begin{align*}
		(j'_n)_\ast(1)&=(j'_n)_\ast (p_n)^\ast(1)\\
		&=\pi^\ast (s_0)_\ast(1)\\
		&=(p_{n+1})^\ast(s_0)^\ast(s_0)_\ast(1)\\
		&=t_1. 
	\end{align*}
	Here, $s_0:\Gm/\mu_p\to \A^1\times_{\mu_p}\Gm$ denotes the map induced zero section. The second equality follows from the functoriality of the Gysin sequences. Thus, the right most vertical arrow in the Gysin diagram is given by the multiplication by $t_1$, finishing the proof.
\end{proof}

The following was used in the proof
\begin{lemma}
	\label{lem: helplemma commutativediagram}
	The following diagram is commutative and the outer square is a pullback square in the category of schemes over $\mathrm{Spec}(k)$
	\[\begin{tikzcd}
		{U_n\times_{\mu_p}\Gm} & {U_{n+1}\times_{\mu_p}\Gm} \\
		{\Gm/\mu_p} & {\A^{1}\times_{\mu_p}\Gm}.
		\arrow["(j'_n)"{inner sep=.8ex}, "\shortmid"{marking}, hook, from=1-1, to=1-2]
		\arrow["{p_n}"', from=1-1, to=2-1]
		\arrow["{p_{n+1}}"{description}, from=1-2, to=2-1]
		\arrow["\pi", from=1-2, to=2-2]
		\arrow["{s_0}"'{inner sep=.8ex}, "\shortmid"{marking}, hook, from=2-1, to=2-2]
	\end{tikzcd}\]
	Here $\pi$ is induced from the $\mu_p$-eqivariant morphism $U_{n+1}\times\Gm\to \Gm\times \Gm\to\A^1\times\Gm$ given by first projecting onto the last factor and then including into $\A^1$. In particular, we obtain identities
	\begin{align*}
		(p_{n})_\ast&=(p_{n+1})_\ast (j'_n)_\ast,\\
		(j'_n)_\ast (p_n)^\ast&=\pi^\ast (s_0)_\ast.
	\end{align*}
\end{lemma}
\begin{proof}
	This follows directly from the following $\mu_p$-equivariant diagram 
	\[\begin{tikzcd}
		{U_n\times\Gm} & {U_{n+1}\times\Gm} \\
		\Gm & {\A^{1}\times\Gm},
		\arrow["j'_n"{inner sep=.8ex}, "\shortmid"{marking}, hook, from=1-1, to=1-2]
		\arrow["{p_n}"', from=1-1, to=2-1]
		\arrow["{p_{n+1}}"{description}, from=1-2, to=2-1]
		\arrow["\pi", from=1-2, to=2-2]
		\arrow["{s_0}"'{inner sep=.8ex}, "\shortmid"{marking}, hook, from=2-1, to=2-2]
	\end{tikzcd}\]
	where the outer square is a pullback square, as can be seen from the functor of points description using that limits of sheaves can be calculated section-wise. Further the left vertical and the diagonal are the canonical projections, and the right vertical is given by first projecting onto the last factor and then including into $\A^1$. The upper horizontal arrows is the obvious inclusions, while the lower one is given by the constant section at 1.
	
	The in particular part now follows from applying the functoriality of Gysin sequence for closed pair, see also \cite[Proposition 3.1.4]{degliseorientation} for the first equation.
\end{proof}
\begin{proof}[Proof of \Cref{prop: scheme version of quillens technical lemma}]
	Denote by $\pi_{n+1}:U_{n+1}\times_{\mu_p}\Gm\to U_{n+1}/\mu_p$ the projection onto the first factor. The Euler class of the bundle $\A^{n+1}\times_{\mu_p}\Gm\to \Gm/\mu_p$ is, by definition, $t_{n+1}$. We obtain a Gysin sequence
	\[\begin{tikzcd}
		{E^{\ast,\ast}(U_{n+1}\times_{\mu_p}\Gm)} & {E^{\ast,\ast}(U_{n+1}/\mu_p)} & {E^{\ast,\ast}(U_{n+1}/\mu_p)}
		\arrow["{(\pi_{n+1})_{\ast}}", from=1-1, to=1-2]
		\arrow["{t_{n+1}}", from=1-2, to=1-3]
	\end{tikzcd}\]
	By assumption, there exists an element $z\in E^{\ast,\ast}(U_n\times_{\mu_p}\Gm)$ such that 
	\begin{equation*}
		(\pi_{n+1})_\ast z=x\in E^{\ast,\ast}(U_n/{\mu_p}).
	\end{equation*}
	Recall that $j_n:U_n/\mu_p\to U_{n+1}/\mu_p$. From the pullback square
	\[\begin{tikzcd}
		{U_n\times_{\mu_p}\Gm} & {U_{n+1}\times_{\mu_p}\Gm} \\
		{U_n/\mu_p} & {U_{n+1}/\mu_p}
		\arrow["{j_n'}"{inner sep=.8ex}, "\shortmid"{marking}, hook, from=1-1, to=1-2]
		\arrow["{\pi_n}"', from=1-1, to=2-1]
		\arrow["\lrcorner"{anchor=center, pos=0.125}, draw=none, from=1-1, to=2-2]
		\arrow["{\pi_{n+1}}", from=1-2, to=2-2]
		\arrow["{j_n}"'{inner sep=.8ex}, "\shortmid"{marking}, hook, from=2-1, to=2-2]
	\end{tikzcd}\]
	it follows that $(j_n)^\ast x=(j_n)^\ast(\pi_{n+1})_\ast z=(\pi_n)_\ast(j_n')^\ast z$. But by \Cref{lem: gysin sequences} the right most vertical morphism in the diagram of Gysin sequences associated to the above morphism of closed pairs, is 0. Therefore $(p_n)_\ast(j'_n)^\ast z=0$, and hence there exists an element $z'\in E^{\ast,\ast}(\Gm/\mu_p)$, such that $(p_n)^\ast z'=(j'_n)^\ast z$. As the map $\Gm/\mu_p\to\ast\simeq \mu_p/\mu_p$ has a section $i:\ast\to\Gm/\mu_p$ induced by the inclusion $\mu_p\hookrightarrow\Gm$, the Gysin sequence of the closed pair $(\Gm/\mu_p,\ast)$ splits into short exact sequences, and we can write $z'= y'\cdot 1 + y\cdot i_\ast 1$ with $y'\in E^{\ast+1,\ast}$ and $y\in E^{\ast,\ast}$.
	
	We now notice that by the Gysin sequence $(\pi_n)_\ast (p_n)^\ast 1=(\pi_n)_\ast1=(\pi_n)_\ast(\pi_n)^\ast 1 =0$, using that pullbacks are ring homomorphisms. Thus, using that every object we currently dealing with is a module over the coefficient ring $E^{\ast,\ast}$, and all the maps are module maps, we can calculate
	\begin{equation*}
		(j_n)^\ast x=(\pi_n)_\ast p_n^\ast z'=(\pi_n)_\ast (p_n)^\ast (y'\cdot 1+y\cdot i_\ast 1)=y' (\pi_n)_\ast (p_n)^\ast(1)+y(\pi_n)_\ast (p_n)^\ast i_\ast (1)=y(\pi_n)_\ast (p_n)^\ast i_\ast (1).
	\end{equation*}
	Thus, to prove \Cref{lem: gysin sequences}, it suffices to calculate $(\pi_n)_\ast (p_n)^\ast i_\ast (1)$. Consider the pullback diagram
	\[\begin{tikzcd}
		{U_{n+1}\simeq U_{n+1}\times_{\mu_p}\mu_p} & {U_{n+1}\times_{\mu_p}\Gm} \\
		\ast\simeq\mu_p/\mu_p & {\Gm/\mu_p}.
		\arrow["\iota", from=1-1, to=1-2]
		\arrow["{\mathrm{pr}}"', from=1-1, to=2-1]
		\arrow["\lrcorner"{anchor=center, pos=0.125}, draw=none, from=1-1, to=2-2]
		\arrow["{p_n}", from=1-2, to=2-2]
		\arrow["i"', from=2-1, to=2-2]
	\end{tikzcd}\]
	It follows that $(p_n)^\ast i_\ast(1)=\iota_\ast (\mathrm{pr})^\ast(1)=\iota_\ast (1)$. The composition $\pi_n\circ \iota$ can be identified with the $\mu_p$-torsor $U_n\to U_n/\mu_p$, so that by \Cref{lem: Cptorsordividedpseries} we obtain 
	\begin{equation*}
		(\pi_n)_\ast (p_n)^\ast i_\ast (1)=(\pi_n)_\ast\iota_\ast(1)=(f_n\sslash\mu_p)(1)=\langle p\rangle_E(t_n),
	\end{equation*}
	finishing the claim.
\end{proof}

\begin{remark}
	One can also formulate Quillen's technical lemma purely in the language of stacks, see \Cref{Prop: MGL-cohomologyforBetmup}. In this language, \Cref{lem: Cptorsordividedpseries} would translate in to an equivalent statement about the bundle $\ast\to\Bet{\mu_p}$, and the proof of \Cref{Prop: MGL-cohomologyforBetmup} is considerably shorter and more conceptual. To minimize the usage of stacks, and circumvent the problem of a sparse distribution of references, we decided to follow the more explicit, scheme-focused path.
\end{remark}
\begin{lemma}
	\label{Prop: MGL-cohomologyforBetmup}
	Let $E$ be an incoherently normed, oriented spectrum. Denote by $\Bet{\mu_p}$ the étale classifying space of $\mu_p$ torsors, and by $t$ the euler class of the canoncial line bundle $\A^1\sslash\mu_p\to\Bet{\mu_p}$. There is an exact sequence
	\begin{equation*}
		\ldots\to E^{\ast,\ast}\xrightarrow{\langle p\rangle_E(t)}E^{\ast,\ast}(\Bet{\mu_p})\xrightarrow{t}E^{\ast+2,\ast+1}(\Bet{\mu_p})\to\ldots.
	\end{equation*}
	In particular, if there is an element $x\in E^{\ast,\ast}(\Bet{\mu_p})$ such that $tx=0$, then there exists an element $y\in E^{\ast,\ast}(\ast)$, satisfying
	\begin{equation*}
		y\langle p\rangle_E(t)=x.
	\end{equation*}
\end{lemma}
	\subsection{Quillen's Lemma}
We can now finally prove Quillen's lemma.
\begin{lemma}[Quillen]
	\label{lem: Quillenslemma}
	Let $k$ be a field containing a primitive $p$-th root of unity. Let $X\in\Spc{k}_\bullet$ be a compact, pointed motivic space, such that $\Sigma^\infty_+X$ is a finite cell complex. Then there exists a formula for $x\in\mathrm{MGL}^{2q,q}(X)$:
	\begin{equation*}
		\left. \begin{matrix}
			-x&\ \text{for }q>0\\
			x^p-x&\text{for }q=0
		\end{matrix}\right\rbrace \in p(\mathrm{MGL}^{2q,q}(X)_{(p)})+L_{>0}\mathrm{MGL}^{2\ast,\ast}(X)_{(p)}.
	\end{equation*}
	Here, $L$ is the Lazard ring, and we understand $(\mathrm{MGL})^{2\ast,\ast}(X)_{(p)}$ as a graded module over $L$, with module structure induced from the universal orientation on $\MGL$.
\end{lemma}
\begin{proof}
	The proof follows now closely the topological proof as in  \cite[Theorem 2.6.3]{FGBO}.
	
	By Proposition \ref{Prop: GeneralformulaforPCp} for $m\gg0$, there exists a formula for $x\in\widetilde{\mathrm{MGL}}^{2q,q}(X)$
	\begin{align}
		\begin{split}
			\label{eq: initialformula}
			w^mP_{C_p}(x)&= \sum_{|\alpha|\leq q+m}w^{q+m-|\alpha|}a_\alpha(t) s_\alpha(x)\\
			&=w^{q+m}x+ \underset{|\alpha|\neq 0}{\sum_{|\alpha|\leq q+m}}w^{q+m-|\alpha|}a_\alpha(t) s_\alpha(x).
		\end{split}
	\end{align}
	in $\widetilde{\MGL}^{2(qp+m-1),qp+m-1}(X\times\Bet{\mu_p})$. Using Totaro's geometric model $\Bet{\mu_p}\cong\mathrm{colim}_n (\A^n\setminus \{0\})/\mu_p$, we can further pullback to $\bar{U}_n:=(\A^n\setminus\{0\})/\mu_p$ for all $n$ and hence obtain the same relation as above in $\widetilde{\MGL}^{2(qp+m-1),qp+m-1}(X\times \bar{U}_n)$ for all $n$, with $t$ replaced by $t_n$, the Euler class of the canonical bundle over $\bar{U}_n$. To unify notation, we will from now on set $t:=t_n$, where the $n$ where the $n$ of course depends on the context.
	
	Recall that $w$ is a power series with coefficients in the subring $L\subseteq\pi_\ast(\MGL)_\ast$, where, by abuse of notation, we identified the Lazard ring with its image in $\pi_\ast(\MGL)_\ast$. The leading term of $w$ is $(p-1)!t^{p-1}$. After localizing at $p$, multiplication by $(p-1)!$ acts invertibly. Thus, we may write $t^{p-1}=w\theta(t)$, where $\theta(t)$ is a multiplicativly invertible power series with coefficients in $L$.
	
	Set $N=(m-q)(p-1)$. For all $m$ we obtain an equation
	\begin{equation}
		\label{eq: Fundermentalequationforquillen}
		t^{N}(w^{q}P_{C_p}(x)-x)=\theta(t)^{N}\underset{|\alpha|\neq 0}{\sum_{|\alpha|\leq q+m}}w^{q+m-|\alpha|}a_\alpha(t) s_\alpha(x)=:\psi_x(t)
	\end{equation}
	in $\widetilde{\MGL}^{2(qp+m-1),qp+m-1}(X\times \bar{U}_{m})$, where $\psi_x$ is a power series with coefficients in $L_{>0}(\widetilde{\mathrm{MGL}})^{2\ast,\ast}(X)_{(p)}$. In particular, this equation is true for $m=N+1$
	
	Consider now the set of all possible $N>0$ for which there exists a formula of the form
	\begin{equation*}
		t^{N}(w^{q}P_{C_p}(x)-x)=\psi_x(t)\in \widetilde{\MGL}^{2(qp+m-1),qp+m-1}(X\times \bar{U}_{N+1})
	\end{equation*}
	as in \eqref{eq: Fundermentalequationforquillen}, that is with $\psi_x$ a power series with coefficients in $L_{>0}(\widetilde{\mathrm{MGL}})^{2\ast,\ast}(X)_{(p)}$. This set is non-empty as it contains at least $N=(m-q)(p-1)$. Let $N$ be the minimal such value. Pulling back along the inclusion at the base point $X\to X\times \bar{U}_n$ in particular sets $t=0$, as under this pullback, the canonical line bundle on $\bar{U}_n$ becomes trivial, hence the top Chern class vanishes. So, from the above formula, we obtain $\psi(0)=0$ and hence $\psi_x(t)=t\varphi(t)$ is at least once divisible by $t$. We obtain the new equation
	\begin{equation*}
		t(t^{N-1}(w^{q}P_{C_p}(x)-x)-\varphi(t))=0.
	\end{equation*}
	Due to our setup, this equation holds in $(\widetilde{\MGL})^{2\ast,\ast}(X\times \bar{U}_{N+1})_{(p)}$, so we can apply \Cref{prop: scheme version of quillens technical lemma} to find an element $y\in(\widetilde{\MGL})^{2\ast,\ast}(X)_{(p)}$, such that
	\begin{equation*}
		t^{N-1}(w^{q}P_{C_p}(x)-x)=\varphi(t)+\langle p\rangle(t)y\in \widetilde{\MGL}^{2\ast,\ast}(X\times \bar{U}_{N})_{(p)}.
	\end{equation*}
	If now $N>1$, then $y\in L_{>0}(\widetilde{\MGL})^{\ast,\ast}(X)$ for degree reasons. So that we obtain a new equation of the form\footnote{Notice the drop in the exponent of $t$ as well as in the index of $\bar{U}_N$. This is the reason why we had to choose $m$ to be at least $N+1$, so that we can do this part of the induction.} 
	\begin{equation}
		t^{N-1}(w^{q}P_{C_p}(x)-x)=\psi'_x(t)\in \widetilde{\MGL}^{2(qp+m-1),qp+m-1}(X\times \bar{U}_{N}).
	\end{equation}
	As $N-1<N$, this contradicts the minimality assumption on $N$, hence $N=1$. We obtain, by pulling back once again along the inclusion $X\to X\times \bar{U}_N$, the expression
	\begin{equation*}
		\left. \begin{matrix}
			-x&\ \text{for }q>0\\
			x^p-x&\text{for }q=0
		\end{matrix}\right\rbrace = py+ \varphi(0)\in p(\widetilde{\MGL})^{2q,q}(X)_{(p)}+L_{>0}(\widetilde{\MGL})^{2\ast,\ast}(X)_{(p)}.
	\end{equation*}
\end{proof}
	\section{\texorpdfstring{The homotopy groups of $\MGL$}{The homotopy groups of \MGL}}
In this section we calculate the homotopy groups of $\MGL^\wedge_p:=\MGL^\wedge_p(k)\simeq \MGL(k)^\wedge_p$, for $k$ a field containing all primitive $p$-th power roots.

\begin{proposition}
	\label{prop: existence of morphism}
	Fix a prime $p$, and let $k$ be a field containing all primitive $p$-power roots of unity. There exists a morphism
	\begin{equation}
		\label{RtoMGLiso}
		R_{\ast}(k)_{\ast}:=(L)_\ast\otimes_{\Z}\pi_{\ast}(\HZp)_\ast\cong(L^\wedge_p)_\ast\otimes_{\mathbb{Z}^\wedge_p} \KM{\ast}{k}^\wedge_p\otimes_{\mathbb{Z}^\wedge_p} \mathbb{Z}^\wedge_p\left[\tau\right]\to \pi_{\ast}(\MGL^\wedge_p)_\ast
	\end{equation}
	unique up to a choice of Tate twist $\tau\in\pi_1(\MGL^\wedge_p)_1$, see \Cref{prop: tatetwistinmgl}.
\end{proposition}

\begin{proof}
	By the isomorphism \eqref{MGLMilnorK} there is a canonical morphism 
	\begin{equation*}
		\KM{\ast}{k}^\wedge_p\cong \pi_{0}(\MGL^\wedge_p)_{\ast}\to \pi_{\ast}(\MGL^\wedge_p)_{\ast}.
	\end{equation*}
	As $\MGL^\wedge_p$ is oriented by the canonical morphism $\MGL\to\MGL^\wedge_p$. Hence, by the universal property of the Lazard ring, there is a canonical morphism 
	\begin{equation*}
		L_\ast\to\pi_{\ast}(\MGL^\wedge_p)_{\ast}.
	\end{equation*}
	All these morphisms are morphisms of bigraded rings, where we consider $\KM{\ast}{k}$ to be concentrated in the negative diagonal of $\pi_{\ast}(\MGL)_{\ast'}$, and $\left(L\right)_\ast$ in the positive degrees $(2\ast,\ast),\ast\geq0$. 
	
	By the universal property of the tensor product it therefore suffices to construct a morphism 
	\begin{equation*}
		\mathbb{Z}^\wedge_p\left[\tau\right]\to\pi_{\ast}\left(\MGL^\wedge_p\right)_\ast,
	\end{equation*}
	which amounts to constructing an element $\tau\in\pi_{1}(\MGL^\wedge_p)_1$. This was done in \Cref{prop: tatetwistinmgl}, and is where we need that $k$ contains all $p$-power roots of unity.  
\end{proof}

The first step is to show that this morphism is an isomorphism in negative topological degrees.

\begin{lemma}
	\label{isoabovediagonal}
	The morphism
	\begin{equation}
		R_{\ast}(k)_{\ast}\to \pi_{\ast}(\MGL^\wedge_p)_{\ast}
	\end{equation}
	constructed above is an isomorphism in bidegrees $\left(s,\ast\right)$ for $s\leq 0, \ast\in \mathbb{Z}$.
\end{lemma}
\begin{proof}
	For $s=0$, the morphism is an isomorphism by construction, being the inclusion of $\pi_0(\MGL^\wedge_p)_\ast\hookrightarrow \pi_\ast(\MGL^\wedge_p)_\ast$.
	
	For $s\leq -1$, the left hand vanishes by construction, so we only have to show that $\pi_{s}(\MGL^\wedge_p)_\ast$ vanishes to. For this we use the Milnor short exact sequence from \cite[Lemma 7.3.9]{CnossenIntrod}
	\begin{equation*}
		0\to H^1\mathbb{R}\mathrm{Hom}(\Z/p^\infty, \pi_{s}(\MGL)_\ast)\to\pi_{s}(\MGL^\wedge_p)\to H^0\mathbb{R}\mathrm{Hom}(\Z/p^\infty, \pi_{s-1}(\MGL))\to 0.
	\end{equation*}
	By connectivity of $\MGL$, the outer groups vanish again, and hence so does the middle one.
\end{proof}
\begin{remark}
	Notice that for the topological homotopy groups to appear here, we really needed to work on global sections. This is du to the fact that the Milnor short exact sequence as stated in the proof of \Cref{isoabovediagonal} does not exist in this form for motivic spectra in general, but it does for topological spectra, see \cite{mattis2024unstable}.
\end{remark}

For positive degrees the proof becomes quite a bit harder. We will start with injectivity, following closely the idea of how one shows injectivity in the Lazard theorem, see \cite[Lecture 2 and 3]{luriechromatic}, and how Quillen proved that the canonical morphism $L_\ast\to\mathrm{MU}_{\ast}$ is an isomorphism. However, before the actually proof injectivity, we will make a general command on the $p$-completion of Milnor-$K$ theory for the field $k$.
\begin{lemma}
	\label{Lem: MilnorKthyisptorsionfree}
	The $p$-adic completion of the Milnor-K theory $\KM{\ast}{k}^\wedge_p$ is p-torsion free. 
\end{lemma}
\begin{remark}
	This is less trivial then it may seem. One could try to argue that the $p$-th power roots of unity are infinitely $p$-divisible (at least in our case). However, it is an open conjecture by Tate, whether all $p$-torsion of $\KM{\ast}{k}$ is determined by the $p$-th root of unity, at least as long as $\mathrm{char}(k)\neq p$, see \cite[Conjecture 2]{izhb}. Instead we need to utilize the norm-residue isomorphism, see \Cref{Thm: Norm-Residue-Iso}.
\end{remark}
\begin{proof}
	Consider the cofiber sequence of motivic spectra,
	\begin{equation*}
		\HZ^\wedge_p\xrightarrow{\cdot p}\HZ^\wedge_p\to\HZ/p.
	\end{equation*}
	The cofiber sequence induces the upper row in the following commutative diagram, while the Milnor $K$-theory appears due to the naturality of the norm-residue isomorphism:
	\[\begin{tikzcd}[column sep= tiny]
		{H^{n-1,n}\left(k,\mathbb{Z}^\wedge_p\right)} & {H^{n-1,n}\left(k,\mathbb{Z}/p\right)} & {H^{n,n}\left(k,\mathbb{Z}^\wedge_p\right)} & {H^{n,n}\left(k,\mathbb{Z}^\wedge_p\right)} & {H^{n,n}\left(k,\mathbb{Z}/p\right)} \\
		{\left(\KM{n-1}{k}^\wedge_p\right)\left\{\tau\right\}} & {\left(\KM{n-1}{k}/p\right)\left\{\tau\right\}} & {\KM{n}{k}^\wedge_p} & {\KM{n}{k}^\wedge_p} & {\KM{n}{k}/p}
		\arrow[from=1-1, to=1-2]
		\arrow[from=1-2, to=1-3]
		\arrow["\cdot p", from=1-3, to=1-4]
		\arrow[from=1-4, to=1-5]
		\arrow[two heads, from=2-1, to=2-2]
		\arrow["\cong"', from=2-1, to=1-1]
		\arrow["\cong"', from=2-2, to=1-2]
		\arrow["\cong"', from=2-3, to=1-3]
		\arrow["\cdot p", from=2-3, to=2-4]
		\arrow["\cong"', from=2-4, to=1-4]
		\arrow[two heads, from=2-4, to=2-5]
		\arrow["\cong"', from=2-5, to=1-5].
	\end{tikzcd}\]
	
	The morphism $\left(\KM{n-1}{k}^\wedge_p\right)\left\{\tau\right\}\to \left(\KM{n-1}{k}/p\right)\left\{\tau\right\}$ is the natural projection, hence surjective. It follows that the morphism $H^{n-1,n}\left(k,\mathbb{Z}^\wedge_p\right)\to H^{n-1,n}\left(k,\mathbb{Z}/p\right)$ is surjective, so $H^{n,n}\left(k,\mathbb{Z}^\wedge_p\right)\to H^{n,n}\left(k,\mathbb{Z}^\wedge_p\right)$ is injective. This in turn implies that $\KM{n}{k}^\wedge_p\xrightarrow{\cdot p}\KM{n}{k}^\wedge_p$ is injective. This argument is independent of $n\geq 1$, implying the claim.
\end{proof}
This lemma actually shows that $\KM{\ast}{k}^\wedge_p$ is completely $\mathbb{Z}$-torsion free. This follows from $\KM{\ast}{k}^\wedge_p$ being a $\mathbb{Z}_p^\wedge$-module. Hence all other primes except $p$ are invertible, implying that $\KM{\ast}{k}^\wedge_p$ can at most have $p$-torsion, which is taken care of by the lemma.
\begin{proposition}
	\label{prop: Mainmorphismisinjective}
	The morphism \eqref{RtoMGLiso} is injective.
\end{proposition}
\begin{proof}
	Consider the composition map
	\begin{align}
		\begin{split}
			\label{HurewiczMotivic}
			R_{\ast}(k)_{\ast}\to\pi_{\ast}(\MGL^\wedge_p)_{\ast}\to&\pi_{\ast}((\MGL \otimes_{\mathbb{\Sph}}\mathrm{H}\mathbb{Z})^\wedge_p)_{\ast}\\
			\cong& \pi_{\ast}(\HZp)_{\ast}[c_1,c_2,\ldots]\\
			\cong&\KM{\ast}{k}^\wedge_p\left[c_1,\ldots\right]\left[\tau\right]
		\end{split}
	\end{align}
	Here the second map is the Hurewicz morphism, induced by the unit map $u:\Sph\to \mathrm{H}\mathbb{Z}$. The second isomorphism is \cite[Proposition 6.2]{naumann2009motiviclandweberexactness}. The last isomorphism is again \eqref{Normresidue2}. We will show that this composition is injective, which implies the claim.
	
	
	Recall the definition of $R_\ast(k)_{\ast'}$,
	\begin{equation*}
		R_{\ast}(k)_{\ast}:=\left(L^\wedge_p\otimes_{\mathbb{Z}^\wedge_p} \KM{\ast}{k}^\wedge_p\otimes_{\mathbb{Z}^\wedge_p} \mathbb{Z}^\wedge_p\left[\tau\right]\right)_{\ast,\ast}\cong\KM{\ast}{k}^\wedge_p\left[t_1,\ldots\right]\left[\tau\right].
	\end{equation*}
	Here we used Lazard's theorem to rewrite the Lazard ring as a polynomial ring in infinity many variables $t_i$ of bidegree $(2i,i)$. 
	
	By \Cref{prop: tatetwistinmgl}, this composition sends $\tau$ to $\tau$, using that the isomorphism from \cite[Proposition 6.2]{naumann2009motiviclandweberexactness} is induced from the orientation on $\HZ$. Thus, by \Cref{isoabovediagonal}, all that is left is to understand the image of the Lazard generators. However, by \Cref{Lem: MilnorKthyisptorsionfree}, $\KM{\ast}{k}^\wedge_p$ is $\mathbb{Z}$-torsion free, so it suffices to show injectivity after rationalization. The classical discussion from topology, see \cite[Lecture 7, Corollary 6]{luriechromatic}, carries to the motivic setting, see also \cite[section 6.1]{hoy15}. That is, the map $L\to\Z[b_1,\ldots]$, induced by the formal group law $\sum_{n\geq0}b_nx^n\in\Z[b_1,ldots]$ with $b_0=1$, sends the indecomposables $(I/I^2)_{2n}$ injectivity to $(J/J^2)_{2n}\simeq \Z$, where $I$ is the ideal of $L$ generated by all elements of positive degree, And $J$ is the same in $\Z[b_1,\ldots]$, where $\mathrm{deg}(b_i)=2i$. The image of $(I/I^2)_{2n}$ is $p\Z$ if $n+1=q^e$ for some some prime $q$ and some $e\geq 1$, and $\Z$ else. In particular, after $p$-completion, the map is an isomorphism for all $n+1$ but those which are powers of $p$. Hence, after rationalization, the map becomes an equivalence.
	
\end{proof}

Before we continue the proof we make a short remark about nilpotency in (reduced) cohomology theories.
\begin{lemma}
	\label{Lem: Nilpotency}
	Let $E\in\mathrm{CAlg}(h\SH{S})$ be a homotopy commutative ring spectrum, $X\in\Spc{X}_\ast$ a pointed, connected motivic space. Assume that $X$ is of finite cohomological dimension with respect to $\left\lbrace \underline{\pi}_{q}(E)_\ast\right\rbrace_{q\in\Z}$ coefficients, that is $\tilde{H}^p(X,\underline{\pi}^{\ast}(E)_\ast):=\pi_0\map{}{\Sigma^\infty X}{\Sigma^p H(\underline{\pi}^{\ast}(E)_\ast)}=0$ for $p$ big enough. Then the reduced cohomology $\tilde{E}^{\ast,\ast}(X)$ is nilpotent.
	
	Similarly, let $F\in\SH{k}_{\geq 1}$ be a $1$-connective spectrum of finite cohomological dimension with respect to $\left\lbrace \underline{\pi}_{q}(E)_\ast\right\rbrace_{q\in\Z}$ coefficients, that is $H^p(F,\underline{\pi}^{\ast}(E)_\ast):=\pi_0\map{}{F}{\Sigma^p H(\underline{\pi}^{\ast}(E)_\ast)}=0$ for $p$ big enough, admitting a homotopy coalgebra structure. Then $E^{\ast,\ast}(F)$ is nilpotent.
\end{lemma}
\begin{proof}
	Consider the multiplicative Atiyah-Hirzebruch spectral sequence associated to the Postnikov filtration on motivic spectra\footnote{This spectral sequence arises, up to re-indexing, from filtering the mapping spectrum $\map{}{\Sigma X}{E}$ using the Postnikov-filtration of $E$ and using \cite[Section 1.2.2]{lurieha}}
	\begin{equation*}
		E_2^{p,q}=\tilde{H}^p(X,\underline{\pi}_{-q}(E)_{\ast})\Rightarrow \tilde{E}^{p+q,\ast}(X),
	\end{equation*}
	where we take Serre's grading convention, that is, the bidegree of the $r$-th differential is $\mathrm{bideg}(d_r)=(r,1-r)$. By the assumption on the finite cohomological dimension of $X$, the $E_2$-page is bounded in the $p$ direction.
	
	Now, take an element $x\in \tilde{E}^{\ast,\ast}(X)$ represent by an element $\tilde{x}\in E^{p,q}_\infty$ on the $E_\infty$-page. If $(p,q)\neq (0,0)$ there exists some $n$ such that $\tilde{x}^n\in E^{p',q'}_\infty$ and there is no non-zero module on the same diagonal to the right of $E^{p',q'}_\infty\simeq F^p E^{p+q}(X)$, for $\{F^\ast\}$ the target filtration on $E^{p+q}$ of the Atizah-Hirzebruch spectral sequence. Hence, $x^n=0$ in  $E^{p+q}(X)$. If instead $(p,q)=(0,0)$, then 
	\begin{align*}
		E_2^{0,0}\cong \tilde{H}^0(X,\underline{\pi}_0(E)_{\ast})&=\mathrm{ker}\left(H^0(X,\underline{\pi}_0(E)_{\ast})\to H^0(\ast,\underline{\pi}_0(E)_{\ast}\right)
	\end{align*}
	As $X$ is connected, the inclusion of the base point is an isomorphism, hence the kernel is zero, $E^{0,0}_2=0$ and thus $E^{0,0}_\infty =0$, so this case is trivially true.
	
	The proof is the same in the spectra case, using the Atiyah-Hirzebruch spectral sequence
	\begin{equation*}
		E_2^{p,q}=H^p(F,\underline{\pi}_{-q}(E)_{\ast})\Rightarrow E^{p+q,\ast}(F).
	\end{equation*}
	The 1-connective assumption is necessary to ensure triviality of the $(p,q)=(0,0)$ case.
\end{proof}

Having established the injectivity of the morphism from \Cref{prop: existence of morphism}, our next goal is to show surjectivity. The idea is a twofold induction. In the first step, we prove that if the morphism from \Cref{prop: existence of morphism} is an isomorphism in bi-degrees $(s',t)$ for all $t\in \Z$ and all $s'<s$ for some fixed $s$, then so it is in the bi-degree $(s,0)$. This uses Quillen's technical lemma, \Cref{lem: Quillenslemma}.

In a second step, we further prove that it is also an isomorphism in bi-degrees $(s,t)$ for all $t\in\Z$. We do so using the effectivity of $\MGL/p$. In particular, for this part of the proof, we have to consider not the global sections, but the sheaves themselves. This is possible, using an unramifiedness argument together with the fact that everything we have proven for $k$ automatically holds true in any field extension of $k$.

Consider the following commutative diagram, where the upper row is exact
\[\begin{tikzcd}[column sep=small]
	{\pi_{s}(\MGL^\wedge_p)_\ast} & {\pi_{s}(\MGL^\wedge_p)_\ast} & {\pi_{s}(\MGL/p)_\ast} & {\pi_{s-1}(\MGL^\wedge_p)_\ast} & {\pi_{s-1}(\MGL^\wedge_p)_\ast} \\
	&&& {R_{s-1}(k)_\ast} & {R_{s-1}(k)_\ast}
	\arrow["p", from=1-1, to=1-2]
	\arrow[from=1-2, to=1-3]
	\arrow["0", from=1-3, to=1-4]
	\arrow["p", hook, from=1-4, to=1-5]
	\arrow["\cong"', from=2-4, to=1-4]
	\arrow["p", hook, from=2-4, to=2-5]
	\arrow["\cong"{description}, from=2-5, to=1-5]
\end{tikzcd}\]
The two right vertical morphisms are isomorphisms by the induction hypothesis. By construction of $R_\ast(k)_\ast$, and \Cref{Lem: MilnorKthyisptorsionfree}, multiplication by $p$ is injective, hence so is the right most horizontal map in the upper row. Thus, the second horizontal map is the canonical projection, inducing an isomorphism 
\begin{equation}
	\label{eq:MGL/pisquotient}
	\pi_{s}(\MGL/p)_\ast\cong \pi_{s}(\MGL^\wedge_p)_\ast/p.    
\end{equation}
We obtain a canonical factorization
\begin{equation*}
	R_{s}(k)_\ast\to R_s(k)_\ast/p\to \pi_{s}(\MGL/p)_\ast.
\end{equation*}

The next lemma is an auxillary lemma that we will need in the induction step.
\begin{lemma}
	\label{lem: descriptionMGLcohomologyofSnmodp}
	Let $s>1$ and denote by $\Sph^{s,t}/p$ the cofiber of the multiplication by $p$ map on the sphere spectrum. Then there exists an isomorphism 
	\begin{equation*}
		\MGL^{0,0}(\Sph^{s-1,0}/p)\cong \pi_{s}(\MGL/p)_0.
	\end{equation*}
	Furthermore, there are isomorphisms
	\begin{equation*}
		\MGL^{0,0}(\Sph^{1,0})\cong \pi_{1}(\MGL)_0,\quad \pi_{1}(\MGL/p)_0\cong \pi_{1}(\MGL)_0/p
	\end{equation*}
\end{lemma}
\begin{proof}
	We calculate
	\begin{align*}
		\MGL^{0,0}(\Sph^{s-1,0}/p)&\cong\pi_0\map{k}{\mathrm{cofib}(\Sph^{s-1}\xrightarrow{p}\Sph^{s-1,0})}{\MGL}\\
		&\cong\pi_{0}\fib(\map{k}{\Sph^{s-1,0}}{\MGL}\xrightarrow{p}\map{k}{\Sph^{s-1,0}}{\MGL})\\
		&\cong\pi_{0}\map{k}{\Sph^{s-1,0}}{\Sigma^{-1}\mathrm{cofib}(\MGL\xrightarrow{p}\MGL)}\\
		&\cong \pi_{0}\map{k}{\Sph^{s,0}}{\MGL/p}\\
		&\cong \left(\MGL/p\right)^{0,0}(\Sph^{s,0})\\
		&\cong \pi_{s}(\MGL/p)_0.
	\end{align*}
	This shows the first claim. The second isomorphism is true by definition.
	
	For the last isomorphism, notice that the multiplication by $p$ on $\MGL$ induces an exact sequence
	\[\begin{tikzcd}[column sep=small]
		{\pi_{1}(\MGL)_0} & {\pi_{1}(\MGL)_0} & {\pi_{1}(\MGL/p)_0} & {\pi_{0}(\MGL)_0} & {\pi_{0}(\MGL)_0} 
		\arrow["p", from=1-1, to=1-2]
		\arrow[from=1-2, to=1-3]
		\arrow[from=1-3, to=1-4]
		\arrow["p", from=1-4, to=1-5]
	\end{tikzcd}\]
	By \Cref{isoabovediagonal}, the right most map is isomorphic to the multiplication by $p$ on the zeroth-part of Milnor $K$-theory. As $\KM{0}{k}\cong \Z$, multiplication by $p$ is injective, so the second map is a surjection, inducing the desired isomorphism.
\end{proof}

\begin{proposition}
	\label{prop: firstinduction}
	Fix an integer $s>0$. Assume the morphism \eqref{RtoMGLiso} is proven to be an isomorphism for all bidegrees $(p,q)$ with $p<s$ and $q$ arbitrary. Then it is surjective, and hence an isomorphism, in the bidegree $(s,0)$.
\end{proposition}
\begin{proof}    
	Let $x\in \pi_{s}(\MGL^\wedge_p)_0/p\cong\pi_{s}(\MGL/p)_0$ for $s\geq1$, where the isomorphism is \Cref{eq:MGL/pisquotient} for $\ast=0$. By \Cref{lem: descriptionMGLcohomologyofSnmodp}, we can consider $x$ as an element in $\MGL^{0,0}(\Sph^{s-1,0}/p)$ for $s>1$ (in $\MGL^{0,0}(\Sph^{1,0})$ for $s=1$). To unify the notation, we define $Q_s:=\Sph^{1,0}$, if $s=1$, and $Q_s:=\Sph^{s-1,0}/p$, if $s>1$. Then, for all $s\geq 1$, we have $x\in \MGL^{0,0}(Q_s)$.
	
	Applying \Cref{lem: Quillenslemma}, we obtain an equation of the form 
	\begin{align*}
		x-x^p\in &p\cdot (\widetilde{\MGL})^{0,0}(Q_s)_{(p)}+L_{> 0}\cdot (\widetilde{\MGL})^{\ast,\ast}(Q_s)_{(p)}\\
		&\cong p\cdot (\pi_s(\MGL/p)_0)_{(p)}+L_{> 0}\cdot (\pi_\ast(\MGL/p)_\ast)_{(p)}\\
		&\cong p\cdot (\pi_s(\MGL^\wedge_p)_0/p)_{(p)}+L_{> 0}\cdot (\pi_\ast(\MGL^\wedge_p)_\ast/p)_{(p)}\\
		&\cong L_{> 0}\cdot (\pi_\ast(\MGL^\wedge_p)_\ast/p).
	\end{align*}
	Here we used \Cref{lem: descriptionMGLcohomologyofSnmodp} for the first isomorphism, and \Cref{eq:MGL/pisquotient} for the second.
	
	Thus we may write 
	\begin{equation*}
		x=x^p+\sum_{j> 0}\sum_{\lambda\in L_{2j,j}}\lambda z_\lambda\in L_{> 0}\cdot (\pi_\ast(\MGL^\wedge_p)_\ast/p),
	\end{equation*}
	where the $\lambda\in L_{2j,j}$ are $\Z$-module generators of the Lazard ring, $z\in(\pi_{\ast}(\MGL^\wedge_p)_\ast)/p$. For a fixed $\lambda \in L_{2i,i}$, the degrees are given by $\mathrm{deg}(\lambda)=(i,-i)$ and $\mathrm{deg}(z_\lambda)=(s-i,i)$. By induction, as $j>0$, we know that the $z_\lambda$ are polynomials in the generators of the Lazard ring and in $\tau$ with coefficients in $K^{\mathrm{M}}_\ast(k)/p$.
	
	Recursively inserting the equation for $x$ into itself, we obtain, for every $n\geq 1$ an equation of the form 
	\begin{equation*}
		x=x^{p^n}+\sum_{j> 0}\sum_{\lambda\in L_{2j,j}}\lambda\Tilde{z}_\lambda.
	\end{equation*}
	Here the $\Tilde{z}_\lambda$ are powers of polynomials in the generators as described above.
	
	Now, $\MGL^{\ast,\ast}(Q_s)$ is nilpotent by \Cref{Lem: Nilpotency}. Indeed, $Q_s$ is clearly $1$-connective, and by definition
	\begin{equation*}
		H^p(Q_s,\underline{\pi}_{-q}(\MGL)_\ast)=\pi_0\map{}{Q_s}{\Sigma^pH\underline{\pi}_{-q}(\MGL)_\ast}	
	\end{equation*}
	vanishes for $p\neq s$, if $s=1$, and for $p\neq s,s+1$, if $s>1$.
	
	Thus, for $n\gg 0$, we have that $x^{p^n}=0$ in $\widetilde{\MGL}^{\ast,\ast}(Q)$. Thus
	\begin{equation*}
		x=\sum_{j> 0}\sum_{\lambda\in L_{2j,j}}\lambda\Tilde{z}_\lambda\in (L/p)_{> 0}\cdot (\pi_{\ast}(\MGL)_{\ast})_{(p)}/p.
	\end{equation*}
	
	and hence the morphism $ R_s(k)_0/p\to \pi_{s}(\MGL^\wedge_p)_0/p$ is surjective.
	
	By \cite[Theorem 2.33]{nikolaus2016rational}, the homotopy groups of $p$-completed spectra and classical $p$-complete groups are $p$-derived complete. Since classical $p$-completed rings are in particular $p$-derived complete, we can apply the derived Nakayama lemma, \cite[\href{https://stacks.math.columbia.edu/tag/0G1U}{Tag 0G1U}]{stacks-project}. Thus \eqref{RtoMGLiso} is a surjection in the bidegree $(s,0)$.
\end{proof}

Throughout the next proposition, we use that a morphism of unramified sheaves $f:M\to N$ is an isomorphism, if it induces an isomorphism on sections $f:M(L)\xrightarrow{\cong}N(L)$ for any finitely generated separable field extension $L/k$, see \cite[Proposition 2.9 and Corollary 2.10]{bachmann2024strongly}.

\begin{proposition}
	\label{prop: second induction}
	Fix an integer $s>0$. Assume the morphism \eqref{RtoMGLiso} is proven to be an isomorphism in the bidegrees $(p,q)$ with $p<s$ and $q$ arbitrary, and in the bidegree $(s,0)$ for all field extensions of $k$. Then it is surjective, and hence an isomorphism, in all bidegrees  $(s,q)$ with $q\in \Z$ for all field extensions of $k$.
\end{proposition}
\begin{proof}
	By the derived Nakayama lemma \cite[\href{https://stacks.math.columbia.edu/tag/0G1U}{Tag 0G1U}]{stacks-project}, we are reduced to showing that the map $R_s(k)_\ast\to \pi_{s}(\MGL)_\ast$ is a surjection modulo $p$ for all $\ast$. In difference to before, where we only considered the global sections, we will now work with the sheaves themselves. This enables us to utilize effectivity of $\MGL$ and, in particular, of $\MGL/p$. 
	
	For brevity we will write $\KMS{\ast}:=\underline{K}^{\mathrm{M}}_\ast/p$ for the unramified Milnor-$K$-theory sheaf modulo $p$, see \cite[Chapter 2]{Morel2012}.
	
	Consider the canonical morphism $\MGL/p\to\HS{0}{\ast}{\MGL/p}$. By \Cref{isoabovediagonal}, and \Cref{eq:MGL/pisquotient}, we know that $\HS{0}{\ast}{\MGL/p}(L)\simeq\KMS{\ast}(L)$ for all field extensions $L/k$. By unramifiedness\footnote{In fact, we know from \cite[Theorem 5.3.3]{morel2003introduction}, that $\underline{\pi}_0(\MGL)_\ast\cong\underline{K}^{\mathrm{M}}_\ast$.}, we thus know $\HS{0}{\ast}{\MGL/p}\simeq\KMS{\ast}$. As $\MGL\in\SH{k}^\mathrm{eff}$, and since effective spectra are closed under colimits, we know that $\MGL/p\in\SH{k}^\mathrm{eff}$ as well. Hence, the above morphism factors over the connective cover $f_0\KMS{\ast}=\mathrm{H}\mathbb{Z}/p$, see \cite[above Lemma 12]{Bachmann_2017}, yielding a morphism $\MGL/p\to f_0\KMS{\ast}$.
	
	\vspace{1em}
	
	We now define inductively spectra $F^{(s)}\in\SH{k}^{\mathrm{eff}}$ as 
	\begin{itemize}
		\item for $s=0$, we set $F^{(0)}=\MGL/p$;
		\item for $s=1$, we set $F^{(1)}=\mathrm{fib}(\MGL/p\to f_0\HS{0}{\ast}{\MGL/p}=f_0\KMS{\ast})$;
		\item more generally, for $s>2$, we set $F^{(s)}=\mathrm{fib}(F^{(s-1)}\to f_0 \Sigma^{s-1}\HS{s-1}{\ast}{F^{(s-1)}})$.
	\end{itemize}
	Notice that the $F^{(s)}$ is indeed effective as $\SH{k}^{\mathrm{eff}}$, as a localizing subcategory, contains all fibers.
	
	Using the explicit description
	\begin{align*}
		\HSeff{i}{f_0\KMS{\ast}}\simeq\begin{cases}
			\HSeff{0}{\MGL/p}\quad &\text{for }i=0\\
			0 &\text{else},
		\end{cases}
	\end{align*}
	we inductively see from the long exact sequence of effective homotopy sheaves, that $F^{(s)}\in\SH{k}^{\mathrm{eff}}_{\geq s}$ and $\HSeff{i}{F^{(s)}}\simeq \HSeff{i}{\MGL/p}$ for $i\geq s$. That is, $F^{(s)}=\tau^{\mathrm{eff}}_{\geq s}(\MGL/p)$ is the $s$-connective cover in the effective $t$-structure, see \cite[Section 3]{Bachmann_2017} for the definition of the effective $t$-structure and the definition of $\SH{k}^{\mathrm{eff}}_{\geq s}$.
	
	Fixing now $s$ as in the statement of the proposition, we obtain an isomorphism from \Cref{prop: firstinduction}
	\begin{equation*}
		\pi_s(F^{(s)})_0\cong\left(\bigoplus_{\lambda\in L_{2s,s}}k^{\mathrm{M}}_{s}\lbrace\lambda\rbrace\right)\oplus  \left(\bigoplus_{\lambda'\in L_{2(s-1),s-1}}k^{\mathrm{M}}_{s-2}\{\lambda'\tau\}\right)\oplus\ldots\oplus \left(\bigoplus_{\lambda''\in L_{s,s/2}}k^{\mathrm{M}}_{0}\{\lambda''\tau^{s/2}\}\right)
	\end{equation*}
	where $\lambda$ are the $\Z$-module generators of the degree $(2s,s)$ part of the Lazard ring. Notice that the last summand can only appear for $s$ even.
	
	\vspace{1em}
	We claim that there exists a morphism of effective homotopy modules\footnote{Notice that $\KMS{\ast+s}$ for $s>0$ is effective, as it is the cofiber of the multiplication by $\tau^s$ on $\HZ$.}
	\begin{equation*}
		\left(\bigoplus_{\lambda\in L_{2s,s}}\KMS{\ast+s}\right)\oplus\left(\bigoplus_{\lambda\in L_{2(s-1),s-1}}\KMS{\ast+s-2}\right)\oplus\ldots\oplus \left(\bigoplus_{\lambda'\in L_{s,s/2}}f_0\KMS{\ast}\right)\to \HSeff{s}{F^{(s)}}\in\SH{k}^{\mathrm{eff},\heartsuit}
	\end{equation*}
	which induces an isomorphism on $\HS{0}{0}{-}$. By \cite[Proposition 5 (3)]{Bachmann_2017}, the functor $\HS{0}{0}{-}:\SH{k}^{\mathrm{eff},\heartsuit}\to\mathrm{Ab}(\mathrm{Shv}_{\mathrm{Nis}}(\Sm_k))$ is conservative. Thus, the claim implies that the above morphism is itself an isomorphism. Using that the functor $\iota^\heartsuit:\SH{k}^{\mathrm{eff},\heartsuit}\to\SH{k}^{\heartsuit}$ is fully faithful, see \cite[Proposition 5 (2)]{Bachmann_2017}, and the natural isomorphism $\iota^\heartsuit\HSeff{s}{E}\cong \HS{s}{\ast}{\iota E}$ for $E\in \SH{k}^{\mathrm{eff}}_{\geq s}$, we further get an isomorphism 
	\begin{equation*}
		\left(\bigoplus_{\lambda\in L_{2s,s}}\KMS{\ast+s}\right)\oplus\ldots\oplus \left(\bigoplus_{\lambda'\in L_{s,s/2}}f_0\KMS{\ast}\right)\to \HS{s}{\ast}{F}.	
	\end{equation*}
	As the morphism \eqref{RtoMGLiso} was constructed by mapping onto the generators of the Lazard ring and the Tate element, taking global sections results in the desired surjectivity of $R_{1}(k)_{\ast}/p\twoheadrightarrow\pi_{s}(\MGL/p)_\ast$.
	
	We start by constructing the morphism for a given generator $\lambda\in L_{2s,s}$. Notice first that by fully faithfulness of $\iota^\heartsuit:\SH{k}^{\mathrm{eff},\heartsuit}\to\SH{k}^{\heartsuit}$, it suffices to construct a maps $\KMS{\ast+s-2t}\to \HS{s}{\ast}{F^{(s)}}$ for $0\leq t<2s$ with the required properties.
 	
	Next, $\MGL$ is canonically an $\MGL$-module. We inductively see that $f_0 \Sigma^{s-1}\HS{s-1}{\ast}{F^{(s-1)}}$, and hence $F^{(s)}$, carry canonical $\MGL$-module structures. In particular, $F^{(s)}$ is $\eta$-torsion. As $\MGL$ is a connective ring spectrum, the category of modules over $\MGL$ acquires a canonical $t$-structure induced by the homotopy $t$-structure, see \cite[Proposition 7.1.1.13]{lurieha} (2). Thus $\HS{s}{\ast}{F^{(s)}}\in\left(\mathrm{Mod}_{\MGL}\right)^{\heartsuit}\simeq \mathrm{Mod}_{\HS{0}{\ast}{\MGL}}$, see \cite[Proposition 7.1.1.13]{lurieha} (3). But $\HS{0}{\ast}{\MGL}\simeq \underline{K}^{\mathrm{M}}_{\ast}$, so that $\HS{s}{\ast}{F^{(s)}}$ is canonically a $\underline{K}^{\mathrm{M}}_{\ast}$-module, see \Cref{prop: MGLandHZsame}.
	
	Even more is true. We know inductively that $\pi_{s}(\MGL/p)_{\ast}(L)$ is $p$-torsion for any finitely generated, unramified field extension $L/k$, see \eqref{eq:MGL/pisquotient}. Thus, by unramifiedness of the homotopy sheaves, $\HS{s}{\MGL/p}{\ast}\cong \HS{s}{\ast}{F^{(s)}}$ is $p$-torsion. In particular, $\HS{s}{\ast}{F^{(s)}}$ is canonically a $\KMS{\ast}$-module.
	
	For $0\leq t<s/2$ we calculate
	\begin{align*}
		\pi_0\Map{\KMS{\ast}}{\KMS{\ast+s-2t}}{\HS{s}{\ast}{F^{(s)}}}& \cong \pi_0\Map{\KMS{\ast}}{\Sigma^{s-2t}_{\mathbb{G}_m}\KMS{\ast}}{\HS{s}{\ast}{F^{(s)}}}\\
		&\cong \pi_0\Map{\SH{k}}{\Sigma^{s-2t}_{\mathbb{G}_m}\Sph}{\HS{s}{\ast}{F^{(s)}}}\\
		&\cong \pi_0(\HS{s}{\ast}{F^{(s)}})_{-s+2t}\\
		&\cong \pi_s(F)_{-s+2t}
	\end{align*}
	Here the first equivalence is writing the (negative) contraction of a homotopy module as a $\mathbb{G}_m$-suspension, i.e.\! tensoring with $\mathbb{G}_m$, and the second is the free-forgetful adjunction between spectra and $\KMS{\ast}$-modules.
	
	Now from $\HSeff{s}{F^{(s)}}\simeq \HSeff{s}{\MGL/p}$, we obtain 
	\begin{equation*}
		\pi_{s}(F^{(s)})_{-s+2t}\cong \pi_{s}({\MGL/p})_{-s+2t}\supset \bigoplus_{\lambda\in L_{2(s-t),s-t}}k_{0}\lbrace \lambda\tau^{t}\rbrace
		= \bigoplus_{\lambda\in L_{2(s-t),s-t}}\Z/p\lbrace \lambda\tau^{t}\rbrace
	\end{equation*}
	where the direct sum is again indexed over the module generators of the Lazard ring in degree $(2(s-t),s-t)$. The inclusion follows from a contraction argument, see \cite{bachmann2024strongly}, together with the first induction step, see \Cref{prop: firstinduction}. Thus there exists a unique element in $\pi_{s}(F^{(s)})_{-s+2t}$ mapping to a fixed generator $\lambda\tau^t$ in $\pi_{s}(\MGL/p)_{-s+2t}$. By the above calculation, this element defines a morphism $\KMS{\ast+s-2t}\to \HS{s}{\ast}{F^{(s)}}$, which, by construction, hits exactly the summand $k^{\mathrm{M}}_{s-2t}\{\lambda\tau^t\}$ of $\pi_s(F^{(s)})_0$.
	
	\vspace{1em}
	We now turn to the additional contributions coming from the Milnor $K$-theory appearing in even degrees, that is, the additional factors of motivic cohomolgy.
	
	For this term, we fix any $\lambda'\in L_{s,s/2}$ and take the map  which corresponds to the element $\lambda'\tau^{s/2}$ under the chain of isomorphisms
	\begin{align*}
		\pi_0\Map{f_0\KMS{\ast}}{f_0\KMS{\ast}}{\HS{s}{\ast}{F^{(s)}}}&\cong\pi_0\Map{\KMS{\ast}}{\KMS{\ast}}{\HS{s}{\ast}{F^{(s)}}}\\
		&\cong \pi_{s}(F^{(s)})_0\\
		&\cong \pi_s(\MGL/p)_{0}.
	\end{align*}
	The first isomorphism is induced by the base change adjunction along the ring morphism $f_0\KMS{\ast}\to\KMS{\ast}$, which is an isomorphism on $\pi_0(-)_0$. On $\pi_0(-)_0$, the morphism corresponding to $\lambda'\tau^{s/2}$ constructed above then sends $1\in\pi_0(f_0\KMS{\ast})_0$ to $\lambda'\tau^{s/2}$ and is a map of $\pi_0(f_0\KMS{\ast}\ast)_0\cong k_0^{\mathrm{M}}$-modules, hence has as image $k_0^{\mathrm{M}}\{\lambda'\tau^{s/2}\}\subseteq\pi_s(F^{s})_0$, which is exactly what we wanted.
	
	Taking the coproduct over all the above constructed morphisms, the resulting morphism on $\pi_0(-)_0$ is a $k^{\mathrm{M}}$-module morphism, and, by the above calculations, hits all generators of $\pi_s(F^{s})_0$, finishing the claim.
\end{proof}

\begin{theorem}
	Let $k$ be a field containing all primitive $p$-th power roots of unity. Then there is an isomorphism of bigraded rings
	\begin{equation*}
		(L)_\ast\otimes_{\Z}\pi_{\ast}(\HZp)_\ast\cong\KM{\ast}{k}^\wedge_p[a_1,a_2,\ldots,\tau] \xrightarrow{\cong}\pi_\ast(\MGL^\wedge_p)_\ast,
	\end{equation*}
	with $a_i$ a set of Lazard generators.
\end{theorem}
\begin{proof}
	The isomorphism is the morphism from \Cref{prop: existence of morphism}. This is an isomorphism in negative topological degrees by \Cref{isoabovediagonal}, and an isomorphism in positive topological degrees by the combination of \Cref{prop: Mainmorphismisinjective}, \Cref{prop: firstinduction} and \Cref{prop: second induction}.
\end{proof}
	
	\printbibliography
	
\end{document}